\documentclass[preprint,12pt]{elsarticle}  %

\usepackage{amssymb}
\usepackage{amsmath}
\usepackage{amsthm} 
\usepackage{booktabs} 
\usepackage[hidelinks]{hyperref}
\usepackage{graphicx} 
\usepackage{tabularx}   
\usepackage{adjustbox}  
\usepackage{multirow}
\usepackage{subcaption}
\usepackage{float} 
\usepackage{placeins} 

\input{arxiv_submitted_stamp_LAA.tex}

\theoremstyle{definition}
\newtheorem{Def}{Definition}
\theoremstyle{definition}
\newtheorem{The}{Theorem}
\theoremstyle{definition}
\newtheorem{Lem}{Lemma}
\theoremstyle{definition}
\newtheorem{Alg}{Algorithm}
\theoremstyle{definition}
\newtheorem{Pro}{Proposition}
\theoremstyle{remark}
\newtheorem{Rem}{Remark}
\theoremstyle{remark}
\newtheorem{Ex}{Example}
\theoremstyle{remark}
\newtheorem{Cor}{Corollary}
\DeclareFontEncoding{LS1}{}{}
\DeclareFontSubstitution{LS1}{stix2}{m}{n}
\DeclareMathAlphabet{\mathscr}{LS1}{stix2scr}{m}{n}
\newcommand{\err}{\ensuremath{ \mathscr{r} }}

\DeclareRobustCommand{\i}{\ensuremath{\mathbf{i}}}
\DeclareRobustCommand{\j}{\ensuremath{\mathbf{j}}}
\DeclareRobustCommand{\k}{\ensuremath{\mathbf{k}}}
\renewcommand{\Re}{\operatorname{Re}}
\renewcommand{\Im}{\operatorname{Im}}
\newcommand{\RR}{\mathbb{R}}
\newcommand{\CC}{\mathbb{C}}
\newcommand{\HH}{\mathbb{H}}

\DeclareMathOperator{\diag}{diag}
\DeclareMathOperator{\rank}{rank}
\DeclareMathOperator{\rvec}{vec}

\DeclareMathOperator{\Range}{range}
\DeclareMathOperator{\Ker}{ker}
\newcommand{\DOI}[1]{\href{https://doi.org/#1}{DOI:\,\nolinkurl{#1}}}
\journal{Linear Algebra and its Applications}
\biboptions{sort&compress} 

\begin{document}
%
\begin{frontmatter}
\ArxivSubmittedStampOverlay
\title{Computing Left Eigenvalues of Quaternion Matrices}
\author{Michael Sebek} 
\affiliation{organization={Department of Control Engineering, Faculty of Electrical Engineering, Czech Technical University in Prague},
            addressline={Technická 2}, 
            city={Prague},
            postcode={16000}, 
            country={Czech Republic}}        
\ead{michael.sebek@fel.cvut.cz}
\begin{abstract}
We present a practical Newton-based method for computing left eigenvalues of quaternion matrices. It uses only standard real/complex linear-algebra kernels via embeddings and applies to matrices of any size. Extensive tests on literature examples and benchmark ensembles, together with a compact MATLAB reference implementation, demonstrate reproducible, certificate-based computations up to size \(64\times 64\), including the detection of multiple spherical components and non-generic phenomena such as more than \(n\) isolated left eigenvalues and left-spectrum deficiency.
\end{abstract}
\begin{keyword}
quaternionic matrix \sep left eigenvalue \sep Newton method
\MSC[2020] 15A33 \sep 15A18 \sep 65F15 \sep 65H10
\end{keyword}
\end{frontmatter}

\section{Introduction}
\label{sec:introduction}
Quaternion matrices arise naturally whenever one needs an algebraic language for coupled $3$-- and $4$--dimensional phenomena, including spatial rotations and kinematics, polarization and vector-sensor signal processing, and various formulations of control and estimation. From a linear-algebraic viewpoint, $\HH^{n\times n}$ inherits many familiar constructions, yet the noncommutativity of $\HH$ makes even basic spectral questions substantially more delicate than over $\CC$; see, e.g.,
\cite{Ward1997,Rodman2014,ColomboSabadiniStruppa2011} for background and further pointers.

In this paper, we focus on \emph{left} eigenvalues: a quaternion $\lambda\in\HH$ is a left eigenvalue of
$A\in\HH^{n\times n}$ if there exists a nonzero vector $v\in\HH^n$ such that
\begin{equation}\label{eq:intro-left-eig}
A v = \lambda v .
\end{equation}
This relation is consistent with the intrinsic right scaling of eigenvectors ($v\mapsto vq$), which leaves $\lambda$
unchanged. For numerical computation, one must additionally fix a (real) \emph{gauge} for $v$ to make local corrections
well posed.

Unlike over $\CC$, the \emph{left} spectrum is not invariant under quaternionic similarity, and it can contain both isolated
points and continua (most prominently, spherical families); see, e.g., \cite{HuangSo2001,FarenickPidkowich2003}.
For \(2\times2\) matrices, the structure is essentially complete: depending on explicit algebraic conditions, one obtains either two distinct isolated left eigenvalues, a single isolated left eigenvalue (i.e., two coincident ones), or an entire \(2\)--sphere of left eigenvalues \cite{HuangSo2001}.
 In higher dimensions, several frameworks support
qualitative analysis, including characteristic-map constructions that represent $\sigma_\ell(A)$ as a zero set in $\HH\simeq\RR^4$ and illuminate rank-degenerate spectral points
\cite{MaciasVirgosPereiraSaez2014,MaciasVirgosPereiraSaez2014b}; see also the general background
\cite{Zhang1997,Zhang2007,Rodman2014}.

What is much less settled is \emph{how to compute} left eigenvalues reliably for general $A\in\HH^{n\times n}$ when $n$ is moderate or large. Noncommutativity brings several practical obstacles: (i) gauge freedom must be handled explicitly,
(ii) the finite left spectrum can deviate from the generic \(n\)-point picture (it may be left-spectrum deficient or contain more than \(n\) isolated eigenvalues), and
(iii) continuous components imply non-isolated solution sets and rank-deficient Jacobians, so standard local Newton theory and simple ``enumeration'' logic do not apply.
Moreover, eliminations in the spectral parameter typically lead to high-degree nonlinear problems in $\RR^4$ whose
cost and numerical sensitivity can grow rapidly with $n$. A further open issue is \emph{diagnostics}: from numerical data, one would like to distinguish isolated eigenvalues from samples of a spherical component, and to provide certificates that remain
meaningful under ill-conditioning and degeneracy.

The goal of this work is practical: to provide an effective computational method that (i) keeps a quaternionic interface at
the user level, (ii) uses only standard real/complex linear-algebra kernels internally via embeddings, and (iii) returns
reproducible outputs together with certificates and diagnostics that help interpret atypical spectral situations. Our approach
is built around a direct eigenpair formulation: we apply Newton's method to the defect equation $Av-\lambda v=0$ augmented
by a simple gauge constraint. This avoids determinantal elimination in the spectral parameter and leads to correction steps that reduce to structured linear solves in an embedded real/complex representation.

The main contributions of the paper are as follows.
\begin{itemize}
\item We derive a Newton iteration for quaternionic left eigenpairs that explicitly fixes the eigenvector gauge, yielding a square real-analytic system and a transparent Jacobian structure. Under a natural nonsingularity condition, we obtain the expected local quadratic convergence for isolated eigenpairs.

\item We build a robust solver around this Newton core using a trial-based multi-start strategy with acceptance, restart, and de-duplication logic tailored to quaternionic spectra. The intent is not only to obtain approximations, but to obtain \emph{reproducible} candidates together with diagnostics that remain meaningful under ill-conditioning, degeneracy, and
non-isolated behavior.

\item We validate the method on classical low-dimensional cases with known closed-form spectra (including spherical families) and on explicit matrices drawn from the literature, and we complement these checks by Monte Carlo experiments on several standard matrix families (dense i.i.d., Hermitian, triangular, and sparse), illustrating both typical performance and rare
hard instances.

\item We provide a compact, stand-alone MATLAB reference implementation intended to support reuse and independent verification. Extended tables and figures, the full experimental protocol and reported metrics, and additional worked examples are deferred to the supplementary material (Supplement~A and Supplement~B).
\end{itemize}

The paper is organized as follows. Section~\ref{sec_quaternions_matrices} fixes notation and recalls the embedding used internally for the linear-algebra kernels (with additional material collected in \ref{app:complex_embedding}). Section~\ref{sec_left_eigs_background} summarizes the left-eigenvalue background needed later. Sections~\ref{sec_newton_left_eigs}--\ref{sec:refinement}
develop the Newton core and the global solver logic, while Section~\ref{sec:illustrative_nongeneric}
illustrates representative non-generic phenomena that motivate the refinement and certification stages.
\ref{app:implementation} summarizes the implementation and validation details.
Supplement~A contains the complete experimental protocol, per-size tables and benchmark figures, and implementation-facing diagnostics, while
Supplement~B collects additional worked examples (a ``Zoo''), including $K>n$ cases, left-spectrum deficient constructions,
and spherical or mixed configurations.

\section{Quaternions and matrices}
\label{sec_quaternions_matrices}
\subsection{The skew field \texorpdfstring{$\HH$}{H}}
\label{subsec_quaternions}
We briefly fix the notation for quaternions, quaternion matrices, and the real embedding used by the Newton step; see, e.g., \cite{Zhang1997,Rodman2014}.

Throughout the paper, $\HH$ denotes the skew field of real quaternions with units $\i,\j,\k$ and relations $\i^2=\j^2=\k^2=\i\j\k=-1$.
Every quaternion can be written as
\[
q=a+b\i+c\j+d\k,\qquad a,b,c,d\in\RR,
\]
with conjugation $\bar q:=a-b\i-c\j-d\k$, real part $\Re(q):=a$, imaginary part $\Im(q):=b\i+c\j+d\k$, and norm $|q|:=\sqrt{q\bar q}\in\RR_{\ge0}$. For $q\neq 0$, $q^{-1}=\bar q/|q|^2$. We also write $\HH^\times:=\HH\setminus\{0\}$ for the group of units.

\subsection{Quaternion vectors and matrices}
\label{subsec_quaternion_matrices}

We treat $\HH^n$ as a right $\HH$-module, so scalar multiplication is $x\alpha$ and the matrices act as right-linear maps:
$A(x\alpha)=(Ax)\alpha$; see \cite{Rodman2014,Zhang1997}.
Left multiplication $\alpha x$ is nevertheless well-defined componentwise and will be used when needed (in particular in the definition of left eigenvalues).

We write $x^\ast$ for the conjugate transpose and use the standard (right) inner product
\[
\langle x,y\rangle := x^\ast y\in\HH,
\qquad
\|x\|_2^2:=x^\ast x\in\RR_{\ge 0}.
\]
For $A\in\HH^{n\times n}$ we denote by
\[
\ker(A):=\{x\in\HH^n:\ Ax=0\},\quad
\Range(A):=\{Ax:\ x\in\HH^n\}
\]
the right null space and right range, respectively.

We use the rank notion compatible with right-linear maps $x\mapsto Ax$, namely the \emph{column right rank} $\rank_{\mathrm{col,right}}(A)$; equivalently one may use the \emph{row left rank} $\rank_{\mathrm{row,left}}(A)$, and these coincide. The following proposition is also standard, see, e.g., \cite{Zhang1997,Rodman2014}.
\begin{Pro}[Rank--nullity over $\HH$]
\label{pro_rank_nullity}
Let $A\in\HH^{n\times n}$ and set
\[
r:=\rank_{\mathrm{col,right}}(A)=\rank_{\mathrm{row,left}}(A),\qquad m:=n-r.
\]
Then $\dim\ker(A)=m$.
\end{Pro}

\subsection{Real embedding}
\label{subsec_real_embedding}
For Newton steps (Section~\ref{sec_newton_left_eigs}), we solve a real linear system. This is facilitated by a real embedding representing quaternionic left multiplication; see, e.g.,
\cite{Zhang1997,Rodman2014}.

\begin{Def}[Real coordinate map on $\HH^n$]
\label{def_real_vec_map}
Define the real-linear isomorphism $\rvec:\HH^n\to\RR^{4n}$ by stacking real coefficients:
if $x_k=a_k+b_k\i+c_k\j+d_k\k$, then
\[
\rvec(x):=(a_1,b_1,c_1,d_1,\dots,a_n,b_n,c_n,d_n)^\top\in\RR^{4n}.
\]
For $q=a+b\i+c\j+d\k\in\HH$ we also write $\rvec(q):=(a,b,c,d)^\top\in\RR^4$.
\end{Def}

\begin{Def}[Real left-multiplication matrix]
\label{def_Lq}
For $q=a+b\i+c\j+d\k\in\HH$ define $L(q)\in\RR^{4\times 4}$ by
\begin{equation}
\label{eq_Lq}
L(q):=
\begin{bmatrix}
 a & -b & -c & -d\\
 b &  a & -d &  c\\
 c &  d &  a & -b\\
 d & -c &  b &  a
\end{bmatrix}.
\end{equation}
\end{Def}

\begin{The}[Real representation of quaternion multiplication]
\label{th_Lq_properties}
For all $p,q\in\HH$,
\[
\rvec(qp)=L(q)\,\rvec(p).
\]
Moreover, $L(\cdot)$ is a real-algebra homomorphism:
\[
L(p+q)=L(p)+L(q),\qquad L(pq)=L(p)L(q),\qquad L(1)=I_4.
\]
\end{The}

\begin{Def}[Real embedding of quaternion matrices]
\label{def_rho}
For $A=[a_{rs}]\in\HH^{n\times n}$ define $\rho(A)\in\RR^{4n\times 4n}$ as the block matrix
\[
\rho(A):=\bigl[L(a_{rs})\bigr]_{r,s=1}^n,
\]
i.e., the $(r,s)$ block of size $4\times 4$ equals $L(a_{rs})$.
\end{Def}

\begin{The}[Intertwining and algebraic rules for $\rho(\cdot)$]
\label{th_rho_properties}
For all $A,B\in\HH^{n\times n}$, $\alpha\in\RR$, and $x\in\HH^n$,
\[
\begin{array}{ll}
\rvec(Ax)=\rho(A)\,\rvec(x), & \rho(A+B)=\rho(A)+\rho(B),\\[2pt]
\rho(AB)=\rho(A)\rho(B), & \rho(\alpha I)=\alpha I.
\end{array}
\]
In particular, $A$ is invertible over $\HH$ if and only if $\rho(A)$ is invertible over $\RR$.
\end{The}

\begin{The}[Kernel and the factor $4$]
\label{th_rho_kernel_rank}
Let $A\in\HH^{n\times n}$. Then
\[
\rvec\bigl(\ker(A)\bigr)=\ker\bigl(\rho(A)\bigr)\subset\RR^{4n},
\qquad
\dim\ker\bigl(\rho(A)\bigr)=4\,\dim\ker(A).
\]
Equivalently, if $r:=\rank_{\mathrm{col,right}}(A)$ and $m:=n-r$, then
\[
\rank\bigl(\rho(A)\bigr)=4r,\qquad \dim\ker\bigl(\rho(A)\bigr)=4m.
\]
\end{The}
\begin{proof}
By Theorem~\ref{th_rho_properties}, $Ax=0$ if and only if $\rho(A)\rvec(x)=0$, so 
$\rvec(\ker(A))$ $=\ker(\rho(A))$.
If $\ker(A)$ has quaternionic dimension $m$, then $\ker(A)\cong\HH^m$ as a right $\HH$-module, which has real dimension $4m$. Since $\rvec$ is a real-linear isomorphism, $\dim\ker(\rho(A))=4m=4\dim\ker(A)$. The rank identities follow from real rank--nullity in $\RR^{4n}$ and Proposition~\ref{pro_rank_nullity}.
\end{proof}

\section{Left eigenvalues}
\label{sec_left_eigs_background}
\subsection{Left eigenvalues: definition and basic facts}
\label{subsec_left_definition_basic}
We recall the definition of left eigenvalues and collect basic properties used later for Newton formulations and residual certification.

\begin{Def}[Left eigenvalues]
\label{def_left_eigs}
Let $A\in\HH^{n\times n}$.
A quaternion $\lambda\in\HH$ is a \emph{left eigenvalue} of $A$ if there exists
$x\in\HH^n\setminus\{0\}$ such that $Ax=\lambda x$.
The set of all left eigenvalues is denoted by $\sigma_\ell(A)$.
\end{Def}

\begin{Rem}[Distinct left eigenvalues]
Throughout, $\sigma_\ell(A)$ denotes the \emph{set of distinct} left eigenvalues (no multiplicities).
In particular, in the $2\times2$ Huang--So reduction to a quaternionic quadratic equation, the two
solutions may coalesce, yielding a \emph{left-spectrum singleton}.
We do not introduce or use any notion of multiplicity of left eigenvalues in this paper.
\end{Rem}

\begin{The}[Existence of left eigenvalues]
\label{th_wood_existence}
Every $A\in\HH^{n\times n}$ has at least one left eigenvalue, i.e. $\sigma_\ell(A)\neq\emptyset$.
\end{The}
The first proof is due to Wood \cite{Wood1985}; see also \cite{MaciasVirgosPereiraSaez2014}.

\begin{Pro}[Right-scaling invariance and uniqueness of the eigenvalue]
\label{pro_gauge_freedom}
Let $A\in\HH^{n\times n}$ and suppose $Ax=\lambda x$ for some $x\in\HH^n\setminus\{0\}$ and $\lambda\in\HH$.
Then:
\begin{enumerate}
\item for every $q\in\HH^\times$, also $A(xq)=\lambda(xq)$ (a left eigenvector is defined only up to right scaling);
\item if $Ax=\mu x$ for some $\mu\in\HH$, then $\mu=\lambda$ (for a fixed nonzero $x$ the left eigenvalue is unique).
\end{enumerate}
\end{Pro}
\begin{proof}
(1) $A(xq)=(Ax)q=(\lambda x)q=\lambda(xq)$.
(2) If $Ax=\lambda x=\mu x$, then $(\lambda-\mu)x=0$; choose $k$ with $x_k\neq 0$ and right-cancel $x_k$ in $\HH$.
\end{proof}
Consequently, Newton-type methods impose a normalization (gauge) constraint on $x$ to obtain a square Jacobian system (Section~\ref{sec_newton_left_eigs}).

\begin{Pro}[Compactness of the left spectrum]
\label{pro_left_spectrum_compact}
For every $A\in\HH^{n\times n}$, the set $\sigma_\ell(A)$ is compact in $\HH\simeq\RR^4$.
\end{Pro}

\begin{proof}
See Proposition~16 in \cite{MaciasVirgosPereiraSaez2014}.
\end{proof}

\begin{Pro}[Real shifts and real scaling]
\label{pro_left_real_shift_scale}
Let $A\in\HH^{n\times n}$ and $\alpha,\beta\in\RR$ with $\beta\neq 0$.
Then
\[
\sigma_\ell(A+\alpha I)=\sigma_\ell(A)+\alpha,
\qquad
\sigma_\ell(\beta A)=\beta\,\sigma_\ell(A).
\]
\end{Pro}
\begin{proof}
If $Ax=\lambda x$, then $(A+\alpha I)x=(\lambda+\alpha)x$ and $(\beta A)x=(\beta\lambda)x$ because $\alpha,\beta\in\RR$ are central.
The reverse implications are identical.
\end{proof}

\begin{Rem}[Lack of similarity invariance for $\sigma_\ell(A)$]
\label{rem_left_not_similarity_invariant}
Unlike the right spectrum, the left spectrum $\sigma_\ell$ is not invariant under quaternionic similarity in general.
(However, it \emph{is} invariant under \emph{real} similarity, i.e., $A\mapsto S^{-1}AS$ with $S\in\RR^{n\times n}$ invertible.)
See \cite{MaciasVirgosPereiraSaez2014,DeLeoScolariciSolombrino2002}.
\end{Rem}

In structured situations, $\sigma_\ell(A)$ can be read off immediately; this is useful for sanity checks and for
interpreting numerical output.

\begin{Pro}[Diagonal and triangular matrices]
\label{pro_diag_triangular_left}
Let $A\in\HH^{n\times n}$.
\begin{enumerate}
\item If $A=\diag(q_1,\dots,q_n)$ is diagonal, then $\sigma_\ell(A)=\{q_1,\dots,q_n\}$.
\item If $A$ is upper triangular with diagonal entries $a_{11},\dots,a_{nn}$, then $\sigma_\ell(A)\subset\{a_{11},\dots,a_{nn}\}$.
\end{enumerate}
\end{Pro}

\begin{proof}
See, e.g., Example~21 in \cite{MaciasVirgosPereiraSaez2014}; also \cite{Zhang1997,Rodman2014}.
\end{proof}

\subsection{Spectrum geometry and cardinality}
\label{subsec_left_2x2_and_beyond}
For $2\times2$ matrices, the geometry of $\sigma_\ell(A)$ is remarkably rigid: Huang and So proved that $\sigma_\ell(A)$ has either one element, two elements, or infinitely many elements; in the infinite case, it is a $2$-sphere in $\HH$ (an affine translate of a sphere in the purely imaginary subspace). See \cite{HuangSo2001} (and refinements for special families in \cite{MaciasVirgosPereiraSaez2009}).

For $3\times3$ matrices, characteristic-map / Cayley--Hamilton techniques yield explicit maps whose zeros coincide with $\sigma_\ell(A)$ in the studied settings \cite{MaciasVirgosPereiraSaez2014b,MaciasVirgosPereiraSaez2014}.
These tools establish existence and enable the analysis of concrete examples, but they do not provide a complete, size-independent classification of the possible geometries or cardinalities (finite vs.\ continuous components, sharp
upper bounds on the number of isolated points, or a general description of mixed configurations).

For larger $n$, the global structure of $\sigma_\ell(A)$ is far less understood \cite{MaciasVirgosPereiraSaez2014,Rodman2014}. The numerical evidence reported below indicates phenomena beyond the classical low-dimensional picture: many isolated left eigenvalues, mixed ``spherical + isolated'' configurations, and, in some cases, multiple spherical components. This further motivates scalable, direct computational methods.

\subsection{The zero eigenvalue: singularity and multiplicity}
\label{subsec_zero_left_eig_mult}
The eigenvalue $\lambda=0$ is special: it is both a left and a right eigenvalue whenever it occurs, and it encodes
singularity in the usual way.

\begin{Pro}[Zero eigenvalue and singularity]
\label{pro_zero_left_right_singularity}
Let $A\in\HH^{n\times n}$. Then the following are equivalent:
\[
0\in\sigma_\ell(A),\qquad 0\in\sigma_\err(A),\qquad \ker(A)\neq\{0\},\qquad A\ \text{is singular}.
\]
\end{Pro}
\begin{proof}
$0\in\sigma_\ell(A)$ iff $(A-0\cdot I)x=Ax=0$ for some $x\neq0$ iff $\ker(A)\neq\{0\}$ iff $A$ is singular.
The right-eigenvalue statement is identical.
\end{proof}

\begin{Cor}[Geometric multiplicity of the zero eigenvalue]
\label{cor_zero_left_geom_mult}
Let $A\in\HH^{n\times n}$ and let $m:=\dim\ker(A)=n-\rank_{\mathrm{col,right}}(A)$.
Then $0\in\sigma_\ell(A)$ if and only if $m>0$, and the geometric multiplicity of the left eigenvalue $0$ equals $m$.
Moreover, with the real embedding $\rho(\cdot)$,
$\dim\ker\bigl(\rho(A)\bigr)=4m$.
\end{Cor}

\begin{Rem}[Algebraic multiplicity at $\lambda=0$]
\label{rem_zero_alg_mult}
The quantity $m=\dim\ker(A)$ is the geometric multiplicity of the eigenvalue $0$.
As in the complex case, the algebraic multiplicity can be larger when generalized null chains exist.
In this paper, we only use the geometric multiplicity, which is directly accessible via $\ker(A)$ (or via $\rho(A)$).
\end{Rem}

\subsection{Algebraic reductions, and computational approaches}
\label{subsec_left_spectrum_computation}
A basic characterization is $\lambda\in\sigma_\ell(A)\Longleftrightarrow
A-\lambda I$  is singular over $\HH\Longleftrightarrow\rho(A-\lambda I)$ is singular over $\RR,$ where $\rho(\cdot)$ is the real embedding from Section~\ref{subsec_real_embedding}.
However, this does \emph{not} reduce the problem to an ordinary eigenvalue computation for the fixed real matrix $\rho(A)$,
because $\lambda$ does not enter as a scalar shift:
\[
\rho(A-\lambda I)=\rho(A)-I_n\otimes L(\lambda),
\]
so the spectral parameter appears blockwise through the $4\times4$ left-multipli\-cation matrix $L(\lambda)$ (unless $\lambda\in\RR$, when $L(\lambda)=\lambda I_4$). Consequently, determinantal elimination leads to a real-algebraic constraint in the four real components of $\lambda$, but expanding it quickly becomes impractical as $n$ grows.
\begin{Ex}[Determinantal elimination is correct but quickly becomes unwieldy]
\label{ex_det_elimination_awkward}
Let $A=[a_{rs}]\in\HH^{n\times n}$ and write
$\lambda=\lambda_r+\lambda_i\,\i+\lambda_j\,\j+\lambda_k\,\k\in\HH.$
When eliminating the eigenvector as shown above, we arrive at the single real equation
\begin{equation}
\label{eq_det_elimination_big}
P(\lambda_r,\lambda_i,\lambda_j,\lambda_k)
:=\det\!\bigl(\rho(A-\lambda I)\bigr)
=\det\!\Bigl(\rho(A)-I_n\otimes L(\lambda)\Bigr)=0,
\end{equation}
where
\[
\rho(A)=\bigl[L(a_{rs})\bigr]_{r,s=1}^n\in\RR^{4n\times 4n},\quad
L(\lambda)=
\begin{bmatrix}
 \lambda_r & -\lambda_i & -\lambda_j & -\lambda_k\\
 \lambda_i &  \lambda_r & -\lambda_k &  \lambda_j\\
 \lambda_j &  \lambda_k &  \lambda_r & -\lambda_i\\
 \lambda_k & -\lambda_j &  \lambda_i &  \lambda_r
\end{bmatrix}\in\RR^{4\times 4}.
\]
Equivalently, \eqref{eq_det_elimination_big} is the (block) determinantal condition
\[
\det\!\begin{bmatrix}
L(a_{11})-L(\lambda) & L(a_{12})          & \cdots & L(a_{1n})\\
L(a_{21})            & L(a_{22})-L(\lambda)& \cdots & L(a_{2n})\\
\vdots               & \vdots             & \ddots & \vdots\\
L(a_{n1})            & L(a_{n2})           & \cdots & L(a_{nn})-L(\lambda)
\end{bmatrix}=0.
\]
Since every entry of $\rho(A)-I_n\otimes L(\lambda)$ depends \emph{affinely} on
$(\lambda_r,\lambda_i,\lambda_j,\lambda_k)$, the polynomial $P$ has total degree at most $4n$.
In practice, however, expanding or globally solving \eqref{eq_det_elimination_big} rapidly becomes impractical as $n$ grows,
and the zero set may include continua (e.g., spherical components), motivating direct solvers for the coupled system $Ax-\lambda x=0$ under a gauge constraint (Section~\ref{sec_newton_left_eigs}).
\end{Ex}
To avoid working directly with the determinantal condition \eqref{eq_det_elimination_big}---whose explicit expansion
rapidly becomes prohibitive---several authors instead derive \emph{polynomial systems} in the four real unknowns
$(\lambda_r,\lambda_i,\lambda_j,\lambda_k)$ by rewriting $(A-\lambda I)x=0$ in real coordinates and eliminating $x$ in a more structured way; see, e.g., \cite{MaciasVirgosPereiraSaez2014,LiuKou2019}. Such reductions can be effective
for very small sizes, but the computational bottleneck is then a \emph{global} solution of a nonlinear polynomial system (often of high degree, with spurious/complex solutions and increasing numerical fragility as $n$ grows).
In contrast, our Newton formulation keeps the eigenvector $x$ and replaces a single global algebraic elimination problem by a sequence of well-conditioned real linear solves inside a local iteration, combined with multiple initializations to explore $\sigma_\ell(A)$.

In low dimensions and special matrix classes, \emph{characteristic maps} can sometimes avoid a full global solve and enable explicit analysis; notably, there is a complete $2\times2$ classification \cite{HuangSo2001}, and several $3\times3$ settings admit characteristic-function constructions \cite{MaciasVirgosPereiraSaez2014b,MaciasVirgosPereiraSaez2014}.
Closed-form cases such as Proposition~\ref{pro_diag_triangular_left} apply for arbitrary $n$ and are useful as sanity checks.

\subsection{Residual certificates for left eigenvalues}
\label{subsec:residual_certificates}
We use an eigenpair residual (for a specific vector) and an eigenvalue-only residual obtained by minimizing over unit vectors. Both are measured in the Euclidean norm on $\HH^n$ (equivalently, on $\RR^{4n}$ under $\rvec(\cdot)$).

\begin{Def}[Eigenpair residual]
\label{def:res_eigenpair}
Let $A\in\HH^{n\times n}$, $\lambda\in\HH$, and $v\in\HH^n$ with $\|v\|_2=1$. Define
\[
\mathrm{res}(A,\lambda,v) := \|Av-\lambda v\|_2.
\]
\end{Def}

\begin{Def}[Minimal (eigenvalue-only) residual]
\label{def:res_min}
Let $A\in\HH^{n\times n}$ and $\lambda\in\HH$. Define
\[
\mathrm{res}_{\min}(A,\lambda)
:=\min_{\|x\|_2=1}\|Ax-\lambda x\|_2.
\]
\end{Def}
We use these scalar residual quantities as numerical \emph{certificates}; in contrast, the vector quantity $Ax-\lambda x\in\HH^n$ is referred to as the eigenpair \emph{defect}.

\begin{Pro}[Basic relations and singularity test]
\label{pro:res_basic}
Let $A\in\HH^{n\times n}$ and $\lambda\in\HH$. Then:
\begin{enumerate}
\item For any unit vector $v\in\HH^n$, $\ \mathrm{res}_{\min}(A,\lambda)\le \mathrm{res}(A,\lambda,v)$.
\item $\ \lambda\in\sigma_\ell(A)\ \Longleftrightarrow\ \mathrm{res}_{\min}(A,\lambda)=0
\ \Longleftrightarrow\ A-\lambda I \text{ is singular over }\HH.$
\item With the real embedding $\rho(\cdot)$ from Section~\ref{subsec_real_embedding},
\[
\mathrm{res}_{\min}(A,\lambda)=\sigma_{\min}\bigl(\rho(A-\lambda I)\bigr),
\]
where $\sigma_{\min}(\cdot)$ is the smallest singular value of $\rho(A-\lambda I)\in\RR^{4n\times 4n}$.
\end{enumerate}
\end{Pro}

\begin{proof}
(1) is immediate from the minimum definition. (2) holds because $\mathrm{res}_{\min}(A,\lambda)=0$ iff
$(A-\lambda I)x=0$ has a nontrivial solution. For (3), $\rho(A-\lambda I)$ represents the real-linear operator
$x\mapsto (A-\lambda I)x$ and $\rvec$ is an isometry between $\HH^n$ (with $\|\cdot\|_2$) and $\RR^{4n}$ (Euclidean norm),
so the minimum over $\|x\|_2=1$ equals the smallest singular value.
\end{proof}

In numerical experiments we primarily report $\mathrm{res}(A,\lambda_k,v_k)$ for accepted eigenpairs
(\ref{app:reported_metrics}). When a vector-free check is needed, $\mathrm{res}_{\min}(A,\lambda)$ can be evaluated via
$\sigma_{\min}\bigl(\rho(A-\lambda I)\bigr)$.

\subsection{Contrast with right eigenvalues}
\label{subsec_left_right_contrast}
For context only (we do not compute right eigenvalues here), $\mu\in\HH$ is a \emph{right eigenvalue} of $A$
if $Ay=y\mu$ for some $y\in\HH^n\setminus\{0\}$; the corresponding set is $\sigma_\err(A)$. A key qualitative difference from left eigenvalues is the effect of right-scaling the eigenvector: if $Ay=y\mu$ and $q\in\HH^\times$, then
\[
A(yq)=y(\mu q)=(yq)(q^{-1}\mu q).
\]
Thus, the same right-eigenspace can be represented by different quaternions related by
\[
\mu' = q^{-1}\mu q,
\]
a relation often called \emph{similarity} (or conjugacy) in $\HH$. The set
\[
[\mu]:=\{\,s^{-1}\mu s:\ s\in\HH^\times\,\}
\]
is the \emph{similarity class} of $\mu$; it is a single point if $\mu\in\RR$, and otherwise a $2$-sphere determined by $\Re(\mu)$ and $|\Im(\mu)|$ \cite{Zhang1997,Rodman2014}. Accordingly, right eigenvalues are naturally defined only up to similarity, and $\sigma_\err(A)$ is similarity-invariant; it is commonly studied via the complex adjoint map $\chi(\cdot)$
(\ref{app:complex_embedding}) \cite{Zhang1997,Zhang2007,Rodman2014}.

By contrast, left eigenvalues satisfy $Ax=\lambda x$ and are unchanged under right scaling of eigenvectors (Proposition~\ref{pro_gauge_freedom}), but $\sigma_\ell(A)$ is generally \emph{not} invariant under quaternionic similarity
(Rem.~\ref{rem_left_not_similarity_invariant}). This asymmetry is one of the reasons why methods tailored to $\sigma_\err(A)$ do not directly carry over to computing $\sigma_\ell(A)$.

As a minimal sanity check, for $A=[q]\in\HH^{1\times 1}$ one has $\sigma_\ell(A)=\{q\}$, whereas
$\sigma_\err(A)=[q]$ is the full similarity class of $q$.
\label{ex_left_right_1x1}

\section{Newton iteration for left eigenpairs with a gauge constraint}
\label{sec_newton_left_eigs}
Let $A\in\HH^{n\times n}$. The main contribution of this paper is a Newton-type method for computing left eigenpairs
$(\lambda,x)\in\HH\times(\HH^n\setminus\{0\})$ by solving the defect equation
\begin{equation}
\label{eq_left_eig}
Ax-\lambda x = 0.
\end{equation}
By Proposition~\ref{pro_gauge_freedom}, \eqref{eq_left_eig} is invariant under right scaling $x\mapsto xq$.
To obtain a square, smooth system (in real coordinates), we fix this freedom by imposing a gauge constraint and apply Newton's method to the resulting gauged system.

A practical advantage of this formulation is that Newton is applied directly to the \emph{eigenpair defect}
$Ax-\lambda x$ (augmented only by simple gauge constraints), rather than to a determinantal or characteristic-map equation in the spectral parameter alone. This keeps each correction step as a structured linear-algebraic problem in
$(\Delta\lambda,\Delta x)$ (Section~\ref{subsec_newton_real_form}), which in our experience yields more robust local behavior than Newton iterations based on polynomial or characteristic eliminations, even in low-degree cases.

\subsection{Gauge constraint and regauging}
\label{subsec_newton_gauge}
A left eigenvector is not unique: if $Ax=\lambda x$, then $A(xq)=\lambda(xq)$ for any $q\in\HH^\times$ (Proposition~\ref{pro_gauge_freedom}). This right-scaling freedom has four real degrees of freedom (the four real components
of $q$), and without fixing it, the eigenpair equation does not define an isolated root in $(\lambda,x)$.
Consequently, the Jacobian of the raw defect equation $Ax-\lambda x=0$ is typically singular along the scaling orbit.
To obtain a square smooth system in real coordinates and a well-posed Newton correction, we fix the freedom by imposing four real constraints, selecting a \emph{canonical representative} from each orbit $\{xq:\ q\in\HH^\times\}$.

Fix an index $j\in\{1,\dots,n\}$ (a \emph{pivot}). We enforce two conditions:
(i) normalization $\|x\|_2=1$ (one real constraint), and (ii) a \emph{real pivot} $x_j\in\RR$ (three real constraints), with the branch $x_j>0$ chosen to remove the remaining sign ambiguity. This choice is particularly convenient here:
It is constructive (every $x$ with $x_j\neq 0$ can be regauged explicitly), it is numerically stable when $j$ is chosen as an index of a large component of $x$, and its linearization leads to simple constraint rows in the real Newton system (Section~\ref{subsec_newton_real_form}).

The gauge is encoded by a smooth constraint map
\begin{equation}\label{eq0gauge}
g:\HH^n\setminus\{x:\ x_j=0\}\to\RR^4,
\qquad
x\mapsto g(x):=
\begin{bmatrix}
\|x\|_2^2-1\\[0.3mm]
\Im(x_j)
\end{bmatrix}.
\end{equation}
The (smooth) gauge constraint is
\begin{equation}
\label{eq_gauge}
g(x)=0,
\end{equation}
which imposes $\|x\|_2=1$ and $x_j\in\RR$. The additional branch selection $x_j>0$ is enforced by regauging
and is not included in the smooth system.

\begin{Lem}[Regauging to \eqref{eq_gauge}]
\label{lem_regauge}
Let $x\in\HH^n$ and assume $x_j\neq 0$ for some $j$.
Define
\[
q_1:=x_j^{-1}|x_j|\in\HH^\times,\;\; \tilde x:=x q_1,
\;\; q_2:=\|\tilde x\|_2^{-1}\in\RR_{>0},\;\; x^{\mathrm g}:=\tilde x\,q_2.
\]
Then $g(x^{\mathrm g})=0$ and $(x^{\mathrm g})_j\in\RR_{>0}$.
Moreover, if $(\lambda,x)$ satisfies \eqref{eq_left_eig}, then $(\lambda,x^{\mathrm g})$ also satisfies \eqref{eq_left_eig}.
\end{Lem}

\begin{proof}
Right scaling preserves \eqref{eq_left_eig}:
$A(xq)=(Ax)q=(\lambda x)q=\lambda(xq)$ for any $q\in\HH^\times$.
With $q_1=x_j^{-1}|x_j|$ we have $\tilde x_j=x_j q_1=|x_j|\in\RR_{>0}$.
Finally, $q_2\in\RR_{>0}$ normalizes the Euclidean norm, hence $\|x^{\mathrm g}\|_2=1$ and
$(x^{\mathrm g})_j\in\RR_{>0}$, which implies $g(x^{\mathrm g})=0$.
\end{proof}

\begin{Ex}[Concrete regauging]
\label{ex_regauge_concrete}
Let $n=2$ and $x=(1+\j,\ 2\i)^\top$ and choose $j=1$.
Then $x_1=1+\j$ has modulus $|x_1|=\sqrt{2}$ and inverse $(1+\j)^{-1}=\tfrac12(1-\j)$.
Following Lemma~\ref{lem_regauge}, set
\[
q_1:=x_1^{-1}|x_1|=\frac{\sqrt{2}}{2}(1-\j),\;\; \tilde x:=xq_1,
\;\; q_2:=\|\tilde x\|_2^{-1},\;\; x^{\mathrm g}:=\tilde x\,q_2.
\]
A direct multiplication gives
\[
\tilde x_1=(1+\j)q_1=|x_1|=\sqrt{2}\in\RR_{>0},\qquad
\tilde x_2=(2\i)q_1=\sqrt{2}\,(\i+\k),
\]
and since $\|\tilde x\|_2^2=|\tilde x_1|^2+|\tilde x_2|^2=2+4=6$ we have $q_2=1/\sqrt{6}$, hence
\[
x^{\mathrm g}
=\frac{1}{\sqrt{6}}\begin{bmatrix}\sqrt{2}\\ \sqrt{2}(\i+\k)\end{bmatrix}
=\begin{bmatrix}1/\sqrt{3}\\ (\i+\k)/\sqrt{3}\end{bmatrix}.
\]
Therefore $\|x^{\mathrm g}\|_2=1$ and $(x^{\mathrm g})_1=1/\sqrt{3}\in\RR_{>0}$, i.e., $g(x^{\mathrm g})=0$.
\end{Ex}

\begin{Rem}[Gauge as a local cross-section and the role of the pivot]
\label{rem_quaternionic_phase}
The right-scaling action $x\mapsto xq$ ($q\in\HH^\times$) has four real degrees of freedom, and the constraint $g(x)=0$ fixes them by selecting a local cross-section of this action on the open set $\{x:\ x_j\neq 0\}$.
The exclusion of $x_j=0$ is therefore essential: if the pivot entry vanishes, the condition ``$x_j\in\RR_{>0}$'' cannot be enforced, and the gauge ceases to define a smooth local section.
In computations, we pick $j\in\arg\max_k |x_k|$ to keep $|x_j|$ away from zero and to improve conditioning of the linearized constraints; if $|x_j|$ becomes small during the iteration, one may switch the pivot and regauge.
\end{Rem}

\subsection{Quaternion-form Newton correction with gauge constraints}
\label{subsec_newton_derivation}
As discussed in Section~\ref{subsec_newton_gauge}, the eigenpair equation $Ax=\lambda x$ is invariant under the four-real-parameter right scaling $x\mapsto xq$ ($q\in\HH^\times$). The gauge constraint selects a canonical representative on the open set $\{x:\ x_j\neq 0\}$, turning the left eigenpair condition into a square smooth system in real coordinates. We now derive the corresponding Newton correction equations directly in $\HH$-form.

Define the eigenpair \emph{defect} (a vector in $\HH^n$)
\[
d(\lambda,x):=Ax-\lambda x\in\HH^n.
\]
(Scalar residual \emph{certificates} are reported separately as $\mathrm{res}(A,\lambda,v)=\|Av-\lambda v\|_2$ and
$\mathrm{res}_{\min}(A,\lambda)=\sigma_{\min}\!\bigl(\rho(A-\lambda I)\bigr)$.)
Then fix a pivot index $j\in\{1,\dots,n\}$ (kept \emph{fixed} in the local analysis), and use the gauge map \eqref{eq0gauge}.
The associated augmented map is
\begin{equation}
\label{eq_Fdef}
F:\HH\times\bigl(\HH^n\setminus\{x:\ x_j=0\}\bigr)\to \HH^n\times\RR^4,
\;\;
F(\lambda,x):=\bigl(d(\lambda,x),\,g(x)\bigr).
\end{equation}
Thus, solving $F(\lambda,x)=(0,0)$ enforces both the eigenpair condition $d(\lambda,x)=0$ and the gauge $g(x)=0$,
so that $(\lambda,x)$ is a \emph{gauged} left eigenpair.
Conversely, any left eigenpair $(\lambda,x)$ with $x_j\neq0$ can be regauged (Lemma~\ref{lem_regauge})
to satisfy $F(\lambda,x)=(0,0)$ without changing $\lambda$.
All derivatives below are understood with respect to the underlying real coordinates
$\HH\simeq\RR^4$ and $\HH^n\simeq\RR^{4n}$.

\noindent\emph{Jacobian and Newton step.}
Let $(\lambda,x)$ be a current (regauged) iterate and consider increments
$(\Delta\lambda,\Delta x)\in\HH\times\HH^n$.
Newton's method chooses $(\Delta\lambda,\Delta x)$ so that the first-order (Jacobian) approximation of
$F(\lambda+\Delta\lambda,\,x+\Delta x)$ cancels the current defect $F(\lambda,x)$, i.e.,
\begin{equation}
\label{eq_newton_linearization}
DF(\lambda,x)[\Delta\lambda,\Delta x]= -\,F(\lambda,x),
\end{equation}
where $DF(\lambda,x)$ denotes the Fr\'echet derivative of $F$ at $(\lambda,x)$.
In block form,
\[
DF(\lambda,x)[\Delta\lambda,\Delta x]
=
\begin{bmatrix}
(A-\lambda I)\Delta x-(\Delta\lambda)\,x\\[0.5mm]
Dg(x)[\Delta x]
\end{bmatrix},
\]
where $(\Delta\lambda)\,x$ denotes left scalar multiplication of $x$ by $\Delta\lambda\in\HH$.
Thus, the Newton correction consists of
\begin{equation}
\label{eq_newton_core}
(A-\lambda I)\,\Delta x-(\Delta\lambda)\,x \;=\; -\,d(\lambda,x),
\end{equation}
together with the linearized gauge constraint $Dg(x)[\Delta x]= -\,g(x)$.
Since the iterates are kept regauged (so that $g(x)=0$), this reduces to $Dg(x)[\Delta x]=0$,
meaning that $\Delta x$ preserves the gauge \emph{to first order}.

\begin{Lem}[Linearized gauge constraints]
\label{lem_linearized_gauge}
Let $x\in\HH^n$ and $\Delta x\in\HH^n$. Then
\[
Dg(x)[\Delta x]=
\begin{bmatrix}
2\Re(x^\ast \Delta x)\\[0.3mm]
\Im(\Delta x_j)
\end{bmatrix}\in\RR^4.
\]
In particular, if $g(x)=0$, then the gauge is maintained to first order by imposing
\begin{equation}
\label{eq_newton_gauge_lin}
\Re(x^\ast \Delta x)=0,\qquad \Im(\Delta x_j)=0.
\end{equation}
\end{Lem}

\begin{proof}
The gauge map $g(x)$ consists of the normalization constraint $\|x\|_2^2-1=0$ and of the requirement that the selected
scalar entry $x_j$ is real, i.e., $\Im(x_j)=0\in\RR^3$.

Since $\|x\|_2^2=x^\ast x\in\RR$, the directional derivative in the direction $\Delta x$ is
\[
D(\|x\|_2^2)[\Delta x]
= D(x^\ast x)[\Delta x]
= (\Delta x)^\ast x + x^\ast(\Delta x)
= 2\Re(x^\ast\Delta x),
\]
because $(\Delta x)^\ast x=\overline{x^\ast\Delta x}$ and $q+\overline{q}=2\Re(q)$ for any $q\in\HH$.

For the remaining three constraints, note that $x_j\in\HH$ is a scalar quaternion and the map
\[
\Im:\HH\to\RR^3,\qquad a+b\i+c\j+d\k \mapsto (b,c,d),
\]
is linear. Hence
\[
D(\Im(x_j))[\Delta x]=\Im(\Delta x_j)\in\RR^3,
\]
i.e., the last three components of $Dg(x)[\Delta x]$ are exactly the imaginary components of the increment in the selected entry $x_j$.
\end{proof}

\begin{Ex}[Newton equations for $n=1$]
\label{ex_newton_n1}
Let $n=1$ and $A=[q]$. Then $d(\lambda,x)=(q-\lambda)x$.
Any solution has $\lambda=q$ and arbitrary $x\neq 0$; a gauge such as $\|x\|_2=1$ and $x\in\RR_{>0}$ selects the
unique gauged eigenvector $x=1$.
\end{Ex}

\begin{Ex}[A diagonal $2\times2$ sanity check]
\label{ex_newton_diag_2x2}
Let $A=\diag(\alpha,\beta)$ with $\alpha\neq\beta$ in $\RR$ and consider the exact left eigenpair
$(\lambda_\star,x_\star)=(\alpha,e_1)$. With the pivot choice $j=1$, the gauge $g(e_1)=0$ holds.
At $(\lambda,x)=(\alpha,e_1)$, the defect is zero, hence the Newton correction solves a homogeneous system and
$\Delta\lambda=\Delta x=0$.
This is consistent with quadratic convergence at a simple (isolated) root.
\end{Ex}

\subsection{Real-coordinate form and the linear Newton system}
\label{subsec_newton_real_form}
The correction equations derived in Section~\ref{subsec_newton_derivation} are most naturally written in $\HH$-form.
However, they define a \emph{real} Newton step: the unknowns $(\Delta x,\Delta\lambda)$ contain $4n+4$ real degrees of
freedom, the gauge constraints take values in $\RR^4$, and the Jacobian is an $\RR$-linear operator.
Rather than introducing quaternionic linear algebra for such mixed $\HH$/$\RR$ constraints, we pass to real coordinates,
where one Newton step becomes a single square real linear system that can be solved robustly by standard routines.

\noindent\emph{Real representation of $A$ and of the scalar coupling.}
Let $\rho(A)\in\RR^{4n\times4n}$ be as in Def.~\ref{def_rho}. We write
\[
u:=\rvec(\Delta x)\in\RR^{4n},\qquad w:=\rvec(\Delta\lambda)\in\RR^{4}.
\]
Besides the left-multiplication matrix $L(\cdot)$ (Def.~\ref{def_Lq}), we use the real matrix $R(q)\in\RR^{4\times4}$
representing \emph{right} multiplication by $q$, defined by
\[
\rvec(pq)=R(q)\,\rvec(p)\qquad\text{for all }p,q\in\HH.
\]
If $q=a+b\i+c\j+d\k$, then
\[
R(q):=
\begin{bmatrix}
 a & -b & -c & -d\\
 b &  a &  d & -c\\
 c & -d &  a &  b\\
 d &  c & -b &  a
\end{bmatrix}.
\]
Therefore, the map $\Delta\lambda\mapsto(\Delta\lambda)x$ becomes
\begin{equation}
\label{eq_Bx_def}
\rvec\bigl((\Delta\lambda)x\bigr)=
\underbrace{\begin{bmatrix}
R(x_1)\\ \vdots\\ R(x_n)
\end{bmatrix}}_{=:B(x)\in\RR^{4n\times4}}\,
\rvec(\Delta\lambda)
=B(x)\,w .
\end{equation}

\begin{Ex}[The coupling matrix $B(x)$ for a simple $x$]
\label{ex_Bx_simple}
Let $n=2$ and $x=(1,\ \j)^\top$. Then
\[
B(x)=\begin{bmatrix} R(1)\\ R(\j)\end{bmatrix}
=\begin{bmatrix} I_4\\ R(\j)\end{bmatrix}\in\RR^{8\times 4}.
\]
\end{Ex}

\noindent\emph{Real form of the constraints.}
Let $\rvec(x)\in\RR^{4n}$ be the stacked real coefficients of $x\in\HH^n$. Then
\[
\Re(x^\ast\Delta x)=\rvec(x)^\top \rvec(\Delta x)=\rvec(x)^\top u,
\]
and $\Im(\Delta x_j)=0$ selects the three imaginary coordinates of the $j$th quaternion component of $\Delta x$.
Accordingly, the four linearized constraints \eqref{eq_newton_gauge_lin} can be written as
\begin{equation}
\label{eq_Cx_def}
C(x)\,u=0,
\end{equation}
where $C(x)\in\RR^{4\times4n}$ has first row $\rvec(x)^\top$ and the remaining three rows are coordinate selectors
extracting the $(\i,\j,\k)$-entries of the $j$th quaternion in the ordering of Def.~\ref{def_real_vec_map}.

\begin{Ex}[Constraint rows $C(x)$ in a toy case]
\label{ex_Cx_toy}
Let $n=2$, pivot $j=1$, and write $u=\rvec(\Delta x)\in\RR^8$. Then $C(x)u$ consists of:
(i) one row $\rvec(x)^\top u$ enforcing $\Re(x^\ast\Delta x)=0$,
(ii) three rows extracting the imaginary coordinates of $\Delta x_1$.
\end{Ex}

\noindent\emph{The Newton linear system.}
Let $u=\rvec(\Delta x)$ and $w=\rvec(\Delta\lambda)$. Then the quaternionic Newton equations
\eqref{eq_newton_core} together with the linearized gauge constraints \eqref{eq_newton_gauge_lin}
are equivalent to the square real system
\begin{equation}
\label{eq_newton_real_system}
\begin{bmatrix}
\rho(A-\lambda I) & -B(x)\\
C(x)              &  0
\end{bmatrix}
\begin{bmatrix}
u\\ w
\end{bmatrix}
=
-
\begin{bmatrix}
\rvec\bigl(d(\lambda,x)\bigr)\\ 0
\end{bmatrix},
\end{equation}
where the zero in the lower block uses $g(x)=0$ (the iterates are regauged before each linear solve).
System \eqref{eq_newton_real_system} has dimension $(4n+4)\times(4n+4)$ and can be solved by standard real linear-algebra
routines \cite{Zhang1997,Rodman2014,FarenickPidkowich2003}.

\begin{Rem}[Consistency check for the embedding]
\label{rem_newton_real_consistency}
The core Newton equation \eqref{eq_newton_core} in $\HH^n$ reads
\[
(A-\lambda I)\Delta x-(\Delta\lambda)x=-r.
\]
Applying $\rvec(\cdot)$ and using Theorem~\ref{th_rho_properties}, together with the definitions of $B(x)$ and $C(x)$,
gives
\[
\rho(A-\lambda I)u-B(x)w=-\rvec(r),
\]
which is exactly the first block row of \eqref{eq_newton_real_system}.
\end{Rem}

\subsection{Local convergence}
\label{subsec_newton_convergence}
Quadratic convergence of Newton's method is guaranteed only in a neighborhood of an \emph{isolated} solution of the \emph{gauged} system $F(\lambda,x)=(0,0)$, equivalently when the Jacobian of $F$ at the solution is invertible.
This motivates the following simplicity notion.

\begin{Def}[Simple gauged left eigenpair]
\label{def_simple_gauged_left_eig}
Let $(\lambda_\star,x_\star)\in\HH\times\HH^n$ satisfy $Ax_\star=\lambda_\star x_\star$ and $g(x_\star)=0$,
with $(x_\star)_j\neq 0$.
We call $(\lambda_\star,x_\star)$ \emph{simple (gauged)} if:
\begin{enumerate}
\item $\ker(A-\lambda_\star I)=\{x_\star q:\ q\in\HH\}$ (geometric multiplicity $1$);
\item $x_\star\notin \Range(A-\lambda_\star I)$.
\end{enumerate}
\end{Def}

\noindent
Here condition~(1) excludes additional independent eigenvectors, while condition~(2) rules out a
generalized-eigenvector relation; together they ensure that the Jacobian of the gauged system is invertible at
$(\lambda_\star,x_\star)$, which is the standard hypothesis for local quadratic convergence of Newton's method.

\begin{The}[Local quadratic convergence]
\label{th_newton_left_eigs}
Let $A\in\HH^{n\times n}$ and let $(\lambda_\star,x_\star)$ be a simple gauged left eigenpair
(Def.~\ref{def_simple_gauged_left_eig}). Then there exists a neighborhood $\mathcal U$ of $(\lambda_\star,x_\star)$
(in $\RR^{4n+4}$ coordinates) such that, for every initial $(\lambda_0,x_0)\in\mathcal U$ with $(x_0)_j\neq 0$,
the Newton iteration defined by \eqref{eq_newton_core} and \eqref{eq_newton_gauge_lin} is well-defined and,
after regauging each step as in Lem.~\ref{lem_regauge}, converges to $(\lambda_\star,x_\star)$ with quadratic rate.
\end{The}

\begin{proof}
In real coordinates, the Newton system is a square smooth system, so local quadratic convergence follows
from the classical Newton theorem once the Jacobian is invertible.  We therefore show that the \emph{homogeneous} Jacobian
equations admit only the trivial solution.  Condition~(2) will rule out a nonzero $\Delta\lambda$, condition~(1) will force
$\Delta x$ to be a pure right-scaling of $x_\star$, and the linearized gauge constraints will eliminate exactly that scaling.

We work in real coordinates, where $F$ is a $C^\infty$ map $\RR^{4n+4}\to\RR^{4n+4}$: the map
$(\lambda,x)\mapsto Ax-\lambda x$ is bilinear over $\RR$, and $g$ is polynomial in the real components.
Thus, standard local Newton theory applies once the Jacobian is invertible; see, e.g.,
\cite{OrtegaRheinboldt2000,Deuflhard2011,DennisSchnabel1996}.

At a solution $(\lambda_\star,x_\star)$, the homogeneous Jacobian system reads
\[
(A-\lambda_\star I)\Delta x-(\Delta\lambda)x_\star=0,\qquad
\Re(x_\star^\ast\Delta x)=0,\qquad \Im(\Delta x_j)=0.
\]
We show that this system has only $\Delta\lambda=0$ and $\Delta x=0$.

Assume first that $\Delta\lambda\neq 0$.  Then the first equation implies
\[
(\Delta\lambda)x_\star=(A-\lambda_\star I)\Delta x\in \Range(A-\lambda_\star I).
\]
Since every nonzero quaternion is invertible, right-multiplication by $(\Delta\lambda)^{-1}$ gives
$x_\star\in \Range(A-\lambda_\star I)$, contradicting condition~(2).  Hence $\Delta\lambda=0$, and
\[
(A-\lambda_\star I)\Delta x=0.
\]
By condition~(1) (geometric multiplicity one), $\ker(A-\lambda_\star I)=\{x_\star q:\ q\in\HH\}$, so there exists
$q\in\HH$ such that $\Delta x=x_\star q$.

We now use the linearized gauge constraints to show that $q=0$.
From $\Re(x_\star^\ast\Delta x)=0$ and $\Delta x=x_\star q$ we obtain
\[
0=\Re(x_\star^\ast x_\star q)=\Re\bigl((x_\star^\ast x_\star)\,q\bigr)
=\Re\bigl(\|x_\star\|_2^2\,q\bigr)
=\|x_\star\|_2^2\,\Re(q),
\]
hence $\Re(q)=0$ (since $\|x_\star\|_2^2>0$).

Next, the constraint $\Im(\Delta x_j)=0$ gives
\[
0=\Im(\Delta x_j)=\Im\bigl((x_\star)_j\,q\bigr).
\]
Because $g(x_\star)=0$ and $(x_\star)_j\neq 0$, the chosen pivot satisfies $(x_\star)_j\in\RR_{>0}$.
Multiplication by a nonzero real scalar preserves (and only rescales) the imaginary part, so
\[
\Im\bigl((x_\star)_j\,q\bigr)=(x_\star)_j\,\Im(q)=0
\qquad\Longrightarrow\qquad \Im(q)=0.
\]
Thus $\Re(q)=0$ and $\Im(q)=0$, so $q=0$ and consequently $\Delta x=x_\star q=0$.

We have shown that the homogeneous Jacobian system admits only the trivial solution, so the Jacobian of $F$ is invertible at
$(\lambda_\star,x_\star)$. Local quadratic convergence of the (regauged) Newton iteration then follows from the classical
Newton theorem in $\RR^{4n+4}$.
\end{proof}

\subsection{Spherical solution sets}
\label{subsec_newton_spherical}

Theorem~\ref{th_newton_left_eigs} covers the \emph{isolated} case, where the Jacobian of the gauged system is invertible.
This framework does not apply when the left spectrum contains continua (most notably, spherical components). In that situation, the zero set of
\[
F(\lambda,x)=\bigl(Ax-\lambda x,\ g(x)\bigr)
\]
has positive dimension, and consequently $DF$ is rank deficient at every point on that set. Hence, one cannot, in general, expect quadratic convergence to a \emph{uniquely determined} root.

\begin{Pro}[Non-isolated zeros imply Jacobian singularity]
\label{pro_nonisolated_implies_singular}
Assume $F$ is $C^1$ in real coordinates. If $(\lambda_\star,x_\star)$ is not an isolated zero of $F$, then
$DF(\lambda_\star,x_\star)$ is not invertible.
\end{Pro}
\begin{proof}
If $(\lambda_\star,x_\star)$ is not isolated, there exists a differentiable curve
$t\mapsto(\lambda(t),x(t))$ with $(\lambda(0),x(0))=(\lambda_\star,x_\star)$ and $F(\lambda(t),x(t))\equiv 0$
for $t$ in a neighborhood of $0$. Differentiating at $t=0$ yields
\[
DF(\lambda_\star,x_\star)\,[\dot\lambda(0),\dot x(0)]=0,
\]
with $(\dot\lambda(0),\dot x(0))\neq 0$, hence the kernel of $DF(\lambda_\star,x_\star)$ is nontrivial.
\end{proof}

Nevertheless, Newton steps remain valuable as a \emph{local defect-reduction} mechanism: the correction equations reduce $\|F(\lambda,x)\|$ in directions transverse to the solution set, while motion \emph{along} a continuum is not
controlled and is typically determined by the initial guess and algorithmic choices (e.g., damping, stopping rule, and the linear solver). For this reason, we do not claim a uniqueness or quadratic-rate theorem in spherical cases.

At the same time, the same multi-start Newton pipeline that is used to recover isolated eigenvalues can be exploited to \emph{detect spherical components} of $\sigma_\ell(A)$. When a spherical continuum is present, different initializations typically converge to different points on the same manifold. Collecting sufficiently many distinct converged points thus provides samples from the component. In our computations, we post-process these samples in $\RR^4$ by clustering and by robustly fitting a sphere model (center, radius, and the supporting affine $3$-space), with outlier rejection, and we
validate the inliers by residual-based certificates (Section~\ref{subsec:residual_certificates}). This does not give
completeness guarantees---Newton does not ``enumerate'' a continuum---but it yields reproducible numerical evidence for spherical structure, as illustrated in Supplement~B.

\begin{Ex}[A spherical case where Newton is delicate]
\label{ex_newton_spherical_warning}
For
\[
A=\begin{bmatrix}0&1\\ -1&0\end{bmatrix}\in\RR^{2\times2}\subset\HH^{2\times2},
\]
the left spectrum is a $2$-sphere, hence eigenpairs are not isolated; see \cite{HuangSo2001}.
This non-isolated structure is reflected in the Jacobian of the gauged system.

Fix the pivot index $j=1$ and consider the gauged solution
\[
\lambda=\i,\qquad x=\frac1{\sqrt2}\begin{bmatrix}1\\ \i\end{bmatrix},
\]
so that $Ax=\lambda x$, $\|x\|_2=1$, and $x_1\in\RR_{>0}$.
Then the homogeneous Jacobian equations
\[
(A-\lambda I)\Delta x-(\Delta\lambda)x=0,\qquad
\Re(x^\ast\Delta x)=0,\qquad
\Im(\Delta x_1)=0
\]
admit a nonzero solution, for instance
\[
\Delta\lambda=\j,\qquad
\Delta x=\frac1{\sqrt2}\begin{bmatrix}0\\ \j\end{bmatrix}.
\]
Indeed,
\[
(A-\i I)\Delta x
=\frac1{\sqrt2}\begin{bmatrix}\j\\ -\i\j\end{bmatrix}
=\frac1{\sqrt2}\begin{bmatrix}\j\\ -\k\end{bmatrix}
=(\Delta\lambda)x,
\]
and $\Im(\Delta x_1)=0$, while $x^\ast\Delta x=\tfrac12\,\i\j=\tfrac12\,\k$, hence $\Re(x^\ast\Delta x)=0$.
Therefore, the Jacobian of the gauged system is not invertible at $(\lambda,x)$, and quadratic convergence cannot be
expected in general.
\end{Ex}

\subsection{Algorithm}
\label{subsec_newton_algorithm}
\begin{Alg}[Direct Newton iteration for a gauged left eigenpair $Ax=\lambda x$]
\label{alg_newton_left_eig}
\noindent\textbf{Input:} $A\in\HH^{n\times n}$; initial guess $(\lambda^{(0)},x^{(0)})$ with $x^{(0)}\neq 0$;
tolerance $\varepsilon$; maximum iterations $K_{\max}$; damping parameter $\alpha\in(0,1]$ (optional).\\
\textbf{Output:} approximate gauged left eigenpair $(\lambda,x)$ and residual $\|Ax-\lambda x\|_2$.
\begin{enumerate}
\item \emph{Initialization and pivot selection:}
set $\lambda\gets \lambda^{(0)}$, $x\gets x^{(0)}$, choose
$j\in\arg\max_k |x_k|$, and regauge $x$ (Lem.~\ref{lem_regauge}) so that $g(x)=0$ and $x_j>0$.
\item \emph{Iterate} for $k=0,1,\dots,K_{\max}-1$:
\begin{enumerate}
\item \emph{Residual check:} set $r\gets Ax-\lambda x$. If $\|r\|_2\le \varepsilon$, terminate.
\item \emph{Newton correction:} solve \eqref{eq_newton_core} together with \eqref{eq_newton_gauge_lin}
(equivalently, solve the real system \eqref{eq_newton_real_system}) for $(\Delta x,\Delta\lambda)$.
\item \emph{Damped update:} set $x\gets x+\alpha\,\Delta x$ and $\lambda\gets \lambda+\alpha\,\Delta\lambda$.
\item \emph{Regauging:} enforce $g(x)=0$ and $x_j>0$ again using Lem.~\ref{lem_regauge}.
Optionally, if $|x_j|$ becomes too small, switch the pivot index and regauge
(for the local theory, keep $j$ fixed).
\end{enumerate}
\item \emph{Return:} $(\lambda,x)$ and $\|Ax-\lambda x\|_2$.
\end{enumerate}
\end{Alg}

\begin{Ex}[One-step Newton sanity check on a shifted diagonal]
\label{ex_newton_shifted_diag}
Let $A=\diag(\alpha,\beta)$ with $\alpha\neq\beta$ in $\RR$. Consider the left eigenpair near
$(\lambda_\star,x_\star)=(\alpha,e_1)$ and choose pivot $j=1$.
If $(\lambda,x)$ is gauged with $x=(x_1,x_2)^\top$ and $x_1\in\RR_{>0}$, then
\[
d(\lambda,x)=\begin{bmatrix}(\alpha-\lambda)x_1\\ (\beta-\lambda)x_2\end{bmatrix}.
\]
Equation \eqref{eq_newton_core} decouples in the second component as
\[
(\beta-\lambda)\Delta x_2-(\Delta\lambda)x_2=-(\beta-\lambda)x_2,
\]
which is solved by $\Delta x_2=-x_2$ when $\Delta\lambda=0$.
Thus, away from the target eigenpair, a Newton step tends to suppress spurious components in the eigenvector.
\end{Ex}

\subsection{Extensions}
\label{subsec_newton_extensions}
\begin{Rem}[Generalized pencils and Rosenbrock-type systems]
\label{rem_pencil_extension}
The same construction applies to generalized left eigenpairs of a quaternionic pencil $M-\lambda E$:
solve $(M-\lambda E)z=0$ and impose a gauge on $z$ (or on $Ez$ when $Ez\neq 0$).
The Newton correction remains affine in $\lambda$,
\[
(M-\lambda E)\Delta z-(\Delta\lambda)\,Ez=-\bigl(M-\lambda E\bigr)z,
\]
together with the same four real constraints that fix the right scaling.
This is particularly convenient for descriptor and Rosenbrock-type linearizations, where the dependence on $\lambda$
is affine by design.
\end{Rem}

\section{Illustrative examples from the literature}
\label{sec:illustrative_literature_examples}
This section collects explicit test matrices from the literature and illustrates the behavior of our Newton-based solver.
Whenever an explicit description of $\sigma_\ell(A)$ is available (e.g., the $2\times2$ cases of Huang and So), we quote it and compare against our computed outputs. Otherwise, we report numerical evidence via residual-based certificates. No implementation details (software, parameter choices) are needed for the
interpretation below.

\subsection{\texorpdfstring{$2\times2$}{2x2} examples of Huang and So}
\label{subsec:examples_huangso}
\begin{Ex}\label{ex:HS_25}
\emph{(Two real left eigenvalues; \cite[Ex.~2.5]{HuangSo2001}).}
Consider
\[
A_{\mathrm{HS},2.5}=
\begin{bmatrix}
0 & 1+\i\\
1-\i & 0
\end{bmatrix}\in\HH^{2\times2}.
\]
Huang and So show that
\[
\sigma_\ell(A_{\mathrm{HS},2.5})=\{+\sqrt{2},-\sqrt{2}\}.
\]
Our multi-start solver returns two approximations that match $\pm\sqrt2$ up to roundoff.
In a representative run, the (eigenvalue-only) certificates over the two returned values are
\[
\begin{array}{rcl}
\mathrm{res}_{\min}(A_{\mathrm{HS},2.5},\lambda)\ \text{(median|max)}
&\approx& 9.70\times10^{-13}\ \big|\ 1.94\times10^{-12},\\[2mm]
\frac{\mathrm{res}_{\min}(A_{\mathrm{HS},2.5},\lambda)}{\max\{1,\|A_{\mathrm{HS},2.5}\|_2\}}\ \text{(median|max)}
&\approx& 2.84\times10^{-13}\ \big|\ 5.68\times10^{-13}.
\end{array}
\]
One computed representative carries spurious $\j$/$\k$ components at the $10^{-12}$ level
(e.g., $+\sqrt2+\mathcal{O}(10^{-12})\j+\mathcal{O}(10^{-12})\k$), while the other is real after rounding and is certified
near machine precision ($\mathrm{res}_{\min}\approx 1.5\times10^{-16}$ in this run), consistent with local quadratic
convergence at isolated solutions.
\end{Ex}

\begin{Ex}\label{ex:HS_26}
\emph{(Two non-real left eigenvalues; \cite[Ex.~2.6]{HuangSo2001}).}
Consider
\[
A_{\mathrm{HS},2.6}=
\begin{bmatrix}
0 & \i\\
\j & 1
\end{bmatrix}\in\HH^{2\times2}.
\]
Huang and So give the explicit left spectrum
\[
\sigma_\ell(A_{\mathrm{HS},2.6})
=\left\{
\frac{1+\i+\j-\k}{2},\;
\frac{1-\i-\j-\k}{2}
\right\}.
\]
Our solver returns these two values up to floating-point roundoff.
In a representative run, the certificates (median|max over the two returned values) are
\[
\begin{array}{rcl}
\mathrm{res}_{\min}(A_{\mathrm{HS},2.6},\lambda)\;
\text{(median|max)}
&\approx &1.34\times10^{-11}\ \big|\ 2.68\times10^{-11},
\\[2mm]
\frac{\mathrm{res}_{\min}(A_{\mathrm{HS},2.6},\lambda)}{\max\{1,\|A_{\mathrm{HS},2.6}\|_2\}}\;
\text{(median|max)}
&\approx& 4.91\times10^{-12}\ \big|\ 9.81\times10^{-12}.
\end{array}
\]
As in Example~\ref{ex:HS_25}, one of the two eigenvalues is typically certified at essentially machine precision
($\mathrm{res}_{\min}\approx 6.3\times10^{-17}$ in this run).
\end{Ex}

\begin{Ex}\label{ex:HS_27}
\emph{(Spherical (infinite) left spectrum; \cite[Ex.~2.7]{HuangSo2001}).}
Consider
\[
A_{\mathrm{HS},2.7}=
\begin{bmatrix}
2 & \i\\
-\i & 2
\end{bmatrix}\in\HH^{2\times2}.
\]
Huang and So show that $\sigma_\ell(A_{\mathrm{HS},2.7})$ is infinite and can be parametrized as
\[
\sigma_\ell(A_{\mathrm{HS},2.7})
=\left\{
2-\beta-\delta\,\j+\gamma\,\k:\ \beta^2+\gamma^2+\delta^2=1
\right\},
\]
i.e., a $2$-sphere in the affine $3$-space $2+\mathrm{span}_\RR\{1,\j,\k\}$.
In this case Newton's method cannot enumerate $\sigma_\ell(A_{\mathrm{HS},2.7})$; instead it converges to a particular
point on the solution manifold.
In a representative multi-start run, two distinct points on the sphere are returned with
\[
\begin{array}{rcl}
\mathrm{res}_{\min}(A_{\mathrm{HS},2.7},\lambda)\ \text{(median|max)}
&\approx& 1.11\times10^{-11}\ \big|\ 1.89\times10^{-11},
\\[2mm]
\frac{\mathrm{res}_{\min}(A_{\mathrm{HS},2.7},\lambda)}{\max\{1,\|A_{\mathrm{HS},2.7}\|_2\}}\ \text{(median|max)}
&\approx& 2.25\times10^{-12}\ \big|\ 3.84\times10^{-12}.
\end{array}
\]
Enabling the sphere-detection stage (sampling + fitting) confirms a spherical component with high reliability:
from $20$ sampled points the fitted parameters are \emph{center} $2$ and \emph{radius} $1$, and after refinement the inlier
certificates satisfy
\[
\mathrm{res}_{\min}(A_{\mathrm{HS},2.7},\lambda)\ \text{(median|max)}
\approx 2.58\times10^{-17}\ \big|\ 1.93\times10^{-16},
\]
with on-sphere deviation at the $10^{-16}$ level (median $\approx 3.33\times10^{-16}$, max  $\approx 7.77\times10^{-16}$).
\end{Ex}

\subsection{\texorpdfstring{$3\times3$}{3x3} examples of Mac\'ias--Virg\'os and Pereira--S\'aez}
\label{subsec:examples_mvps}
\begin{Ex}\label{ex:MVPS_19}
\emph{(Three distinct eigenvalues; \cite[Ex.~19]{MaciasVirgosPereiraSaez2014}).}
Consider
\[
A_{19}=
\begin{bmatrix}
\i & 0 & 0\\
\k & \j & 0\\
-3\i & 2\k & \k
\end{bmatrix}.
\]
The paper states
\[
\sigma_\ell(A_{19})=\{\i,\j,\k\}.
\]
Our computation confirms this: the three distinct eigenvalues $\i,\j,\k$ are recovered up to roundoff.
In a representative run, the eigenvalue-only certificates satisfy
\[
\begin{array}{rcl}
\mathrm{res}_{\min}(A_{19},\lambda)\ \text{(median|max)} 
&\approx& 
1.37\times10^{-18}\ \big|\ 9.18\times10^{-18},
\\[2mm]
\frac{\mathrm{res}_{\min}(A_{19},\lambda)}{\max\{1,\|A_{19}\|_2\}}\ \text{(median|max)}
&\approx&
2.68\times10^{-19}\ \big|\ 1.79\times10^{-18},
\end{array}
\]
while the corresponding eigenpair residuals are at the $10^{-16}$ level
(median $\approx 2.29\times10^{-16}$, max $\approx 3.48\times10^{-16}$), consistent with a well-conditioned isolated case.
\end{Ex}

\begin{Ex}\label{ex:MVPS_38}
\emph{(Pole is not a left eigenvalue; \cite[Ex.~38]{MaciasVirgosPereiraSaez2014}).}
Consider
\[
A_{38}=
\begin{bmatrix}
0 & \i & 1\\
3\i-\k & 0 & 1\\
\k & -1+\j+\k & 0
\end{bmatrix}.
\]
A key qualitative statement in the paper is that the pole $\pi_A=-\i$ is \emph{not} a left eigenvalue.
In a representative run, the computed candidates cluster into two distinct values; importantly, neither is close to $-\i$
(e.g., their distances to $-\i$ in $\RR^4$ are about $7.5\times10^{-1}$ and $1.6$), hence the membership test
``$-\i\notin\sigma_\ell(A_{38})$'' passes.
The certificates remain small but are less uniform than in Example~\ref{ex:MVPS_19}:
\[
\begin{array}{rcl}
\mathrm{res}_{\min}(A_{38},\lambda)\ \text{(median|max)}&\approx& 1.03\times10^{-15}\ \big|\ 1.38\times10^{-13},
\\[2mm]
\frac{\mathrm{res}_{\min}(A_{38},\lambda)}{\max\{1,\|A_{38}\|_2\}}\ \text{(median|max)}&\approx& 1.76\times10^{-16}\ \big|\ 2.53\times10^{-14}.
\end{array}
\]
\end{Ex}

\begin{Ex}\label{ex:MVPS_52}
\emph{(Rank-$3$ differential at $\lambda=0$; \cite[Ex.~52]{MaciasVirgosPereiraSaez2014}).}
Consider
\[
A_{52}=
\begin{bmatrix}
\j & 1 & 0\\
2\i & -\k & 1\\
2-\i-2\j & -1-\j+\k & -\i-\k
\end{bmatrix}.
\]
The paper states that $\lambda=0$ is a left eigenvalue and that the differential of the characteristic map at $\lambda=0$
has real rank $3$.
Numerically, we recover an eigenvalue extremely close to $0$ (in a representative run, $\lambda\approx 6.5\times10^{-6}\k$)
together with two additional distinct eigenvalues.
The near-zero eigenvalue is certified with
\[
\mathrm{res}_{\min}(A_{52},\lambda)\approx 2.99\times10^{-11},
\qquad
\frac{\mathrm{res}_{\min}(A_{52},\lambda)}{\max\{1,\|A_{52}\|_2\}}\approx 6.37\times10^{-12},
\]
while the other two eigenvalues are certified more strongly (down to $3.98\times10^{-17}$ in this run).
Moreover, a numerical rank test confirms rank $3$ at $\lambda=0$, matching the paper.
\end{Ex}

\begin{Ex}\label{ex:MVPS_55}
\emph{(Maximal differential rank; \cite[Ex.~55]{MaciasVirgosPereiraSaez2014}).}
Consider
\[
A_{55}=
\begin{bmatrix}
\k & 0 & 0\\
3\i-\j & -\i & \i\\
1-2\k & \j & -\j
\end{bmatrix}.
\]
The paper gives $\sigma_\ell(A_{55})=\{\k,\,0,\,-\i-\j\}$ and states that the differential has maximal rank at each eigenvalue.
Our computation recovers these three eigenvalues up to roundoff (perturbations at the $10^{-11}$ level in a representative run).
The certificates (median|max over the three values) satisfy
\[
\begin{array}{rcl}
\mathrm{res}_{\min}(A_{55},\lambda)\;
\text{(median|max)}
&\approx& 8.90\times10^{-12}\ \big|\ 2.23\times10^{-11},
\\[2mm]
\frac{\mathrm{res}_{\min}(A_{55},\lambda)}{\max\{1,\|A_{55}\|_2\}}\;
\text{(median|max)}
&\approx& 1.51\times10^{-12}\ \big|\ 4.08\times10^{-12},
\end{array}
\]
and $\lambda=0$ is certified near machine precision ($\mathrm{res}_{\min}\approx 6.1\times10^{-18}$ in this run).
The Jacobian-rank checks agree with the paper (maximal rank at all three points).
\end{Ex}

\begin{Ex}\label{ex:MVPS_56}
\emph{(Left-spectrum deficiency; \cite[Ex.~56]{MaciasVirgosPereiraSaez2014}).}
Consider
\[
A_{56}=
\begin{bmatrix}
-\i-\j & 0 & 0\\
\k & -\i & \i\\
1-\i & \j & -\j
\end{bmatrix}.
\]
The paper states that $\sigma_\ell(A_{56})=\{0,\,-\i-\j\}$ (left-spectrum deficient) and that the differential is maximal at
$\lambda=0$ but \emph{identically zero} at $\lambda=-\i-\j$ (null rank).
Numerically, we recover $\lambda=0$ with a tiny certificate ($\mathrm{res}_{\min}\approx 4.4\times10^{-18}$ in a representative run),
and the remaining candidates appear as very close approximations to $-\i-\j$ whose componentwise deviations are of order $10^{-6}$,
consistent with the reported degeneracy.
Quantitatively, in a representative run, the two such approximations satisfy
\[
\begin{array}{rcl}
\mathrm{res}_{\min}(A_{56},\lambda)\;
\text{(median|max)}
&\approx& 2.99\times10^{-11}\ \big|\ 3.75\times10^{-11},
\\[2mm]
\frac{\mathrm{res}_{\min}(A_{56},\lambda)}{\max\{1,\|A_{56}\|_2\}}\;
\text{(median|max)}
&\approx& 6.78\times10^{-12}\ \big|\ 8.49\times10^{-12}.
\end{array}
\]
The Jacobian-rank test matches the paper (maximal at $0$ and null at $-\i-\j$), and the reduced componentwise accuracy at
the null-rank eigenvalue is consistent with this extreme degeneracy.
\end{Ex}

\subsection{\texorpdfstring{$4\times4$}{4x4} dense example}
\label{subsec_examples_4x4}
\begin{Ex}\label{ex:PanNg2024_circ4}
\emph{(A dense $4\times4$ quaternion circulant matrix; \cite[Example~1]{PanNg2024}).}
Consider the dense quaternion circulant matrix $A\in\HH^{4\times 4}$:
\[
\renewcommand{\arraystretch}{1.0}
\setlength{\arraycolsep}{0pt}
\left[\!
\begin{array}{@{}c@{\,}c@{\,}c@{\,}c@{}}
 -2 + \i + \j + 4\k & 2 + 4\i + \j + \k & 1 + 3\i + 2\j + 2\k & -1 + 2\i + 2\j + 3\k \\
 -1 + 2\i + 2\j + 3\k & -2 + \i + \j + 4\k & 2 + 4\i + \j + \k & 1 + 3\i + 2\j + 2\k \\
 1 + 3\i + 2\j + 2\k & -1 + 2\i + 2\j + 3\k & -2 + \i + \j + 4\k & 2 + 4\i + \j + \k \\
 2 + 4\i + \j + \k & 1 + 3\i + 2\j + 2\k & -1 + 2\i + 2\j + 3\k & -2 + \i + \j + 4\k
\end{array}
\!\right].
\]
The authors use this example to illustrate block diagonalization of quaternion circulant matrices; the \emph{left} eigenvalues
are not reported in \cite{PanNg2024}.

Using our solver, we obtained four distinct left eigenvalues. Rounded to two significant digits, they are
\[
\begin{array}{rcl}
\lambda_1 & \approx& -1.5 - 0.43\i + 1.5\j + 4.6\k \\
\lambda_2 & \approx& -2 - 2\i + 2\k \\
\lambda_3 & \approx& -2.9 + 0.85\i - 1.2\j + 5.1\k \\
\lambda_4 & \approx& -5.2 + 1.5\i - 0.25\j + 1.9\k
\end{array}.
\]
The corresponding residual-based certificates are at the level of $10^{-16}$, consistent with near machine-precision accuracy for
the computed eigenpairs.
\end{Ex}

\section{Refinement of numerically found left eigenvalues}
\label{sec:refinement}
The results reported in Supplement~A were obtained by the reference implementation described in \ref{appB:ref_impl}, which implements the method of this paper.
Although the computed left eigenvalues typically come with small residuals, it is often beneficial to further \emph{polish} them.
This post-processing step improves numerical reproducibility, strengthens the credibility of rare or surprising instances, and, in many cases, brings the residuals down to (near) machine precision, as documented in the examples in the following section.

\subsection{Residual certificates and scale-aware normalization}
\label{subsec:residuals_refinement}
We use the residual-based certificates introduced in Section~\ref{subsec:residual_certificates}.
When $A$ is fixed by context, we abbreviate
$\mathrm{res}(\lambda,v):=\mathrm{res}(A,\lambda,v)$ and
$\mathrm{res}_{\min}(\lambda):=\mathrm{res}_{\min}(A,\lambda)$.

For an eigenpair candidate $(\lambda,v)$ with $\|v\|_2=1$, the (pair) residual is
\[
\mathrm{res}(A,\lambda,v)=\|Av-\lambda v\|_2.
\]
The eigenvalue-only certificate is the minimal residual
\[
\mathrm{res}_{\min}(A,\lambda)=\min_{\|x\|_2=1}\|Ax-\lambda x\|_2
=\sigma_{\min}\!\bigl(\rho(A-\lambda I)\bigr),
\]
where $\rho(\cdot)$ is the fixed real (or complex) embedding used for numerical linear-algebra kernels
and $\sigma_{\min}(\cdot)$ is the smallest singular value in that embedding.

Since scaling matters, we also report certificates normalized by a robust matrix scale,
\[
  s(A):=\max\{1,\|A\|_2\},\qquad
  \mathrm{res}_{\min}^{\mathrm{rel}}(A,\lambda):=\frac{\mathrm{res}_{\min}(A,\lambda)}{s(A)}.
\]
Reporting both absolute and normalized values separates genuine numerical difficulty from trivial rescaling effects.

\subsection{Two-stage polishing (standard post-processing)}
Given a candidate $\lambda_0$ returned by the main algorithm, we apply a light\-weight two-stage refinement strategy that is
standard in numerical linear algebra: first improve the \emph{certificate}, then perform a short local correction on the
eigenpair equations; see, e.g., \cite{GolubVanLoan2013}.

\noindent\emph{Stage 1: certificate-driven refinement in $\lambda$.}
We interpret $\lambda=a+b\i+c\j+d\k$ as a point in $\RR^4$ and (locally) minimize
\[
\phi(\lambda):=\mathrm{res}_{\min}(\lambda)=\sigma_{\min}(\rho(A-\lambda I)),
\]
starting from $\lambda_0$. In practice, this is done by a small number of random perturbations of $\lambda_0$ at a few radii,
followed by a short derivative-free local search initialized from the best sampled point \cite{NelderMead1965}. This stage returns a refined
$\lambda_\star$ and an associated near-minimizing vector $v_\star$ (e.g., a right singular vector corresponding to
$\sigma_{\min}(A-\lambda_\star I)$ in the chosen embedding).

\noindent\emph{Stage 2: local eigenpair correction.}
Starting from $(\lambda_\star,v_\star)$, we perform a small number of damped Newton (or quasi-Newton) correction steps on
\[
F(\lambda,v)=Av-\lambda v=0,\qquad \|v\|=1,
\]
using a simple gauge-fixing convention to remove the (right) scalar ambiguity of $v$. When Stage~1 has already reduced
$\mathrm{res}_{\min}$, this local correction typically converges rapidly and yields residuals close to machine precision.

\subsection{Practical interpretation}
The refinement pipeline improves reproducibility and strengthens evidence for rare instances: a reduction of $\mathrm{res}_{\min}(\lambda)$
certifies proximity to singularity of $A-\lambda I$ in the chosen embedding, while the final reported pair residual
$\mathrm{res}(A,\lambda,v)$ verifies a consistent eigenpair. Thresholds are always interpreted relative to a robust matrix scale, such as
$\max\{1,\|A\|_2\}$.

\section{Illustrative non-generic phenomena}
\label{sec:illustrative_nongeneric}

The primary goal of this paper is to provide an effective computational method rather than a systematic theoretical classification of quaternionic left spectra.
Still, once large-scale computations become practical, it is natural to ask whether the solver can also reveal instructive non-generic behavior.
In this section, we briefly report three representative phenomena observed in targeted experiments: (i) more than $n$ isolated left eigenvalues,
(ii) fewer than $n$ isolated left eigenvalues (left-spectrum deficiency), and (iii) a spherical component coexisting with isolated points.
In all cases, we used the core Newton solver together with the refinement strategies of Section~\ref{sec:refinement}.
Additional diagnostics and reproduction scripts are provided in Supplement~B.

\begin{Ex}[\emph{A $3\times3$ integer matrix with five distinct isolated left eigenvalues.}]
\label{ex:byprod_3x3_five}
The following $3\times3$ integer quaternion matrix exhibits \emph{five} distinct isolated left eigenvalues:
\[
A=
\begin{bmatrix}
-7 + 6\i - 6\j - 7\k & 3 - 7\i + 11\j - 2\k & 11 - 9\i \\
6 + \i - 5\j + 3\k & 9 + 7\i - 6\j - 10\k & -5 + 15\i + 14\j - \k\\
16 + 6\i + 14\j + 11\k & 20 - 3\i + 4\j + 6\k & -5 + 19\i + \j - 3\k
\end{bmatrix}.
\]
After refinement, five distinct candidates are
\[
\begin{aligned}
\lambda_1 &= -22.877487 + 15.850469\i - 17.069787\j - 11.791606\k,\\
\lambda_2 &= \phantom{-}11.833188 + \phantom{0}9.698189\i - 13.382634\j - 19.325731\k,\\
\lambda_3 &= \phantom{-}13.399540 + 15.934883\i - 12.000914\j - \phantom{0}0.414566\k,\\
\lambda_4 &= \phantom{-}14.897483 + 16.835221\i - 11.965564\j - \phantom{0}2.863713\k,\\
\lambda_5 &= \phantom{-}21.109974 + 21.579378\i + \phantom{0}5.435201\j - \phantom{0}2.138868\k,
\end{aligned}
\]
with eigenvalue-only certificates
\[
\mathrm{res}_{\min}(\lambda_k)\in[\,1.394\times10^{-16},\,3.716\times10^{-15}\,],
\]
and minimum pairwise separation (in the $\RR^4$ component metric) of approximately
\[
\min_{i\neq j}\|\lambda_i-\lambda_j\|_{\RR^4}\approx 3.00899.
\]
\end{Ex}

\begin{Ex}[\emph{Less than $n$ isolated left eigenvalues (left-spectrum deficiency).}]
\label{ex:byprod_4x4_two}
The following $4\times4$ quaternion matrix was obtained during targeted searches for left-spectrum deficient cases.
When requesting $K=4$ eigenpairs, repeated multi-start runs consistently converged (up to clustering) to only two isolated solutions:
\[
A=
\begin{bmatrix}
1 + 3\i + \k & 8 + \i + \j - 4\k & -8 + 4\k & -7 - 2\i + 4\j + \k \\
-\i - \j & 1 + \i - 2\j + \k & 0 & \i + \j \\
-\i - \j & -1 - 6\j + 3\k & 2 + \i + 4\j - 2\k & 2 + \i + \j - \k \\
0 & 8 - 4\k & -8 + 4\k & -6 + \i + 4\j + 2\k
\end{bmatrix}.
\]
A refined list of distinct candidates is
\[
\lambda_1 = 1 + 2\i - \j + \k,
\qquad
\lambda_2 = -2 + \i + 4\j,
\]
with certificates
\[
\mathrm{res}_{\min}(\lambda_1)\approx 1.07\times10^{-13},
\qquad
\mathrm{res}_{\min}(\lambda_2)\approx 1.69\times10^{-13}.
\]
Thus, the finite left spectrum has cardinality $2<4$, i.e., $A$ is left-spectrum deficient.
In practice, distinctness counting is sensitive to the duplicate threshold: with very tight tolerances (e.g., $10^{-8}$ in $\RR^4$), numerical outputs may form two small clouds around $\lambda_1$ and $\lambda_2$ and can be over-counted; a slightly looser tolerance (e.g., $10^{-5}$) collapses these clouds into the two representatives above.
\end{Ex}

\begin{Ex}[\emph{A spherical component together with isolated points.}]
\label{ex:byprod_spherical}
Spherical components did not arise spontaneously in our random benchmark suites, but can be constructed. One example is the following $4\times4$ matrix $A_{\mathrm{sph}}$:
{\small
\[
\setlength{\arraycolsep}{2pt}\renewcommand{\arraystretch}{1.05}
\begin{bmatrix}
2+5\i-5\j+6\k & 12-3\i-5\j+4\k & 8-\i-\j-2\k & 20-4\i+2\j+2\k\\
4\j & 10+4\i-2\j+4\k & 4\j & 8\j\\
8-\i+3\j-2\k & 28-5\i-3\j & 2+5\i-\j+6\k & 20-4\i+10\j+2\k\\
-4\j & -20+4\i-6\j-2\k & -4\j & -10+8\i-16\j+2\k
\end{bmatrix}
\]}
A run returned $K_{\mathrm{tot}}=20$ distinct candidates; after refinement,
$\mathrm{med}\mathrm{res}_{\min}(\lambda)$ $\approx 2.97\times10^{-16}$ and
$\max \mathrm{res}_{\min}(\lambda)\approx 1.60\times10^{-15}$.
With $\|A_{\mathrm{sph}}\|_2\approx 55.23$, this corresponds to
$\mathrm{med}\,(\mathrm{res}_{\min}(\lambda)/\|A_{\mathrm{sph}}\|_2)\approx 5.37\times10^{-18}$ and
$\max (\mathrm{res}_{\min}(\lambda)/$ $\|A_{\mathrm{sph}}\|_2)$ $\approx 2.89\times10^{-17}$.

A spherical cluster with $18$ inliers was detected, with center $c=10+4\i-6\j+4\k$ and radius $r=8$ in the affine $3$-space
$\{\lambda\in\HH:\ \mathrm{coeff}_{\j}(\lambda)=-6\}$, i.e.,
\[
\mathcal S=\Bigl\{\,c+8\bigl(u_0+u_1\i+u_2\k\bigr):\ u_0^2+u_1^2+u_2^2=1\,\Bigr\}.
\]
Two representative refined inliers are
$\lambda^{(2)}_{a}\approx 4.95 + 7.76\i - 6\j + 8.93\k$ and
$\lambda^{(2)}_{b}\approx 2.57 + 6.92\i - 6\j + 3.50\k$,
while the remaining two candidates were classified as isolated and were very close to
$-6+6\i-4\j+8\k$ and $-10+8\i-8\j+2\k$ (full precision and diagnostics in Supplement~B).
\end{Ex}

\section*{Data and code availability}
\label{sec:data_code_availability}
All numerical results reported in this paper were generated by the authors using an in-house MATLAB benchmarking framework.
To support transparency and practical reuse—without releasing the full experiment infrastructure—we provide a compact,
stand-alone MATLAB reference implementation of the proposed Newton-based method in the public GitHub repository
\href{https://github.com/michaelsebek/leigqNEWTON-public}{\texttt{michaelsebek/leigqNEWTON-public}}.
The repository contains the solver, residual-certificate validation for all accepted eigenpairs, user-facing documentation, and a curated bundle of reproducible examples.
An immutable archival snapshot of the software is deposited on Zenodo (DOI: \url{https://doi.org/10.5281/zenodo.18410141}).

The public package is intentionally minimal: it contains a single main solver function together with a small set of examples and self-tests. It enables readers to (i) run the method on user-supplied quaternion matrices, and
(ii) validate computed eigenpairs via residual-based certificates.
The released code depends only on MATLAB and its built-in \texttt{quaternion} data type; no third-party quaternion toolboxes are required.

The large-scale benchmarking framework used to generate the tables and figures in this article is not part of the public distribution, because it relies on a broader internal research toolbox and auxiliary scripts that are outside the scope of a stable long-term public release.

\section{Conclusions}
We developed a practical Newton-based framework for computing left eigenvalues of quaternion matrices.
The method targets \emph{isolated} left eigenpairs and proceeds one eigenpair at a time, combining a simple gauge that fixes the right scalar ambiguity of eigenvectors with a real-embedded Newton correction assembled via the embedding $\rho(\cdot)$.
As a result, each correction step reduces to a single square real linear system and can be handled by standard real linear algebra, while inputs and outputs remain quaternionic. Beyond isolated eigenvalues, repeated multi-start runs provide certified samples that enable reliable detection and
parameterization of spherical components of $\sigma_\ell(A)$ (Supplement~B).

On the theoretical side, we proved local quadratic convergence for \emph{simple gauged} isolated eigenpairs, i.e., precisely when the Jacobian of the gauged system is invertible. At the same time, non-isolated solution sets (most notably spherical components of the left spectrum) necessarily lead to rank-deficient Jacobians, so Newton cannot be expected to ``enumerate''
a continuum or to converge quadratically to a uniquely determined root. Nevertheless, in such cases, multi-start runs typically converge to multiple distinct points on the solution manifold; collecting these samples and applying a lightweight post-processing step (clustering and sphere fitting in $\RR^4$) provides reproducible numerical evidence for spherical
components, supported by residual-based certificates.

Computationally, the resulting multi-start pipeline (trial--accept--de-du\-pli\-cate, optionally followed by polishing) produced highly accurate certificates once convergence is achieved, across the tested matrix families and sizes. To support reproducibility and reuse, we provide a compact MATLAB reference implementation together with complete supplementary material (documenting the experimental protocol, full tables and figures) and a \emph{Zoo} of illustrative non-generic examples.
Beyond serving as regression tests, this Zoo offers a first glimpse of the variety of spectral phenomena that can arise for larger quaternion matrices, which become accessible once large-scale left eigenvalue computations are practical.

\appendix

\section{Complex embedding}
\label{app:complex_embedding}
This appendix records an \emph{alternative} route to express the Newton correction using the classical complex adjoint (symplectic) embedding $\chi(\cdot)$. While our reference implementation uses the real embedding $\rho(\cdot)$, the same correction can be assembled in complex blocks; we sketch the conversion and explain why we do not adopt it as the reference.

\subsection{Definition and the key equivalence}
Fix the complex slice $\CC_\i:=\{a+b\i:\ a,b\in\RR\}\subset\HH$. Every $A\in\HH^{n\times n}$ can be written as
$A=X+Y\j$ with $X,Y\in\CC_\i^{n\times n}$; see, e.g., \cite{Zhang1997,Rodman2014}.

\begin{Def}[Complex adjoint / symplectic embedding]
\label{def:chi}
For $A=X+Y\j$ with $X,Y\in\CC_\i^{n\times n}$ define
\[
\chi(A):=
\begin{bmatrix}
X & Y\\
-\overline{Y} & \overline{X}
\end{bmatrix}\in\CC^{2n\times 2n}.
\]

\end{Def}

Write $x\in\HH^n$ uniquely as $x=u+v\j$ with $u,v\in\CC_\i^n$, and define the real-linear map
\[
\phi:\HH^n\to\CC^{2n},\qquad
\phi(u+v\j):=\begin{bmatrix}u\\-\overline{v}\end{bmatrix}.
\]

\begin{Pro}[Minimal properties used in this paper]
\label{pro:chi_minimal_app}
For all $A,B\in\HH^{n\times n}$,
\[
\chi(A+B)=\chi(A)+\chi(B),\qquad \chi(AB)=\chi(A)\chi(B),\qquad \chi(A^\ast)=\chi(A)^\ast.
\]
Moreover, for every $A\in\HH^{n\times n}$ and $x\in\HH^n$,
\begin{equation}
\label{eq:phiAx_chi}
\phi(Ax)=\chi(A)\,\phi(x).
\end{equation}
Consequently, the quaternionic linear system $Ax=b$ is equivalent to the complex linear system
$\chi(A)\phi(x)=\phi(b)$ (and hence $A$ is invertible over $\HH$ if and only if $\chi(A)$ is invertible over $\CC$).
\end{Pro}

\begin{proof}
Standard; see \cite{Zhang1997,Rodman2014}. Identity \eqref{eq:phiAx_chi} follows by writing $A=X+Y\j$, $x=u+v\j$ and using
$\j z=\overline{z}\,\j$ for $z\in\CC_\i$.
\end{proof}

\subsection{Sketch: complex-embedded Newton correction}
The quaternion-form Newton correction (with gauge constraint) solves
\[
(A-\lambda I)\Delta x-(\Delta\lambda)x=-d,\qquad Dg(x)[\Delta x]=0,
\]
for $(\Delta\lambda,\Delta x)\in\HH\times\HH^n$, where $d:=Ax-\lambda x$ is the current eigenpair defect.
Transporting the first (vector) equation through $\phi(\cdot)$ and $\chi(\cdot)$ gives the $2n$-dimensional complex-block form
\begin{equation}
\label{eq:chi_newton_core}
\chi(A-\lambda I)\,\phi(\Delta x)-\phi\bigl((\Delta\lambda)x\bigr)=-\phi(d),
\end{equation}
while the gauge contributes four additional real scalar equations.

The only subtlety is the left-multiplication term $(\Delta\lambda)x$. Writing $\Delta\lambda=\delta_1+\delta_2\j$ and
$x=u+v\j$ with $\delta_1,\delta_2\in\CC_\i$ and $u,v\in\CC_\i^n$, one computes
\[
(\Delta\lambda)x=(\delta_1 u-\delta_2\overline{v})+(\delta_1 v+\delta_2\overline{u})\j,
\qquad
\phi\bigl((\Delta\lambda)x\bigr)=
\begin{bmatrix}
\delta_1 u-\delta_2\overline{v}\\
-\overline{\delta_1}\,\overline{v}-\overline{\delta_2}\,u
\end{bmatrix}.
\]
Thus \eqref{eq:chi_newton_core} is \emph{real-linear} (as it must be), but \emph{not complex-linear} in $(\delta_1,\delta_2)$ because
conjugates of the unknowns appear. A purely complex-linear solve would therefore require augmenting the unknowns by their conjugates, which removes the putative advantage of working in complex arithmetic.

In summary, the $\chi$-based formulation is perfectly valid and equivalent to the real-embedding route, but it introduces extra bookkeeping (maintaining a fixed slice $\CC_\i$, consistent $(u,v)$ representations, and real/complex coupling with the gauge). The real embedding $\rho(\cdot)$ yields a single uniform real system with simpler assembly and diagnostics, so we use
$\rho(\cdot)$ in the main text and in the reference implementation.

\section{Implementation and performance}
\label{app:implementation}
This appendix provides a compact summary of the released reference implementation and of the numerical validation.
Complete per-size tables, additional figures, and the full experiment protocol are provided in Supplement~A.

\subsection{Reference implementation}
\label{appB:ref_impl}
A compact, stand-alone MATLAB reference implementation of the proposed Newton-based method is publicly available
\linebreak[2]
(see \emph{Data and code availability} section).
The released package is intentionally minimal: it provides a single main solver routine for user-supplied quaternion matrices together with residual-certificate validation for all accepted eigenpairs; no third-party quaternion toolboxes are required.

The solver keeps a quaternionic user interface (inputs/outputs in $\HH^{n\times n}$ and $\HH^{n}$) but executes each Newton correction using standard real/complex linear-algebra kernels (solves/factorizations) applied to embedded matrices. The reference implementation uses the real embedding $\rho(\cdot)$; an alternative route via the complex adjoint $\chi(\cdot)$
(\ref{app:complex_embedding}) is also supported. Certificates are evaluated in the Euclidean norm on $\HH^n$ (equivalently, the $\ell_2$ norm of stacked real components), hence they are independent of the internal embedding used for the
linear solve.

To compute multiple eigenpairs, the solver performs independent Newton \emph{trials}. Each trial generates an initial guess $(\lambda^{(0)},x^{(0)})$, applies a bounded number of Newton steps, and returns either a candidate or a failure. A candidate is accepted only if its certificate falls below a prescribed threshold; accepted eigenpairs are then de-duplicated (up to a user tolerance) until the requested target $K$ is reached.

Two practical shortcuts are available. For diagonal/triangular inputs, the default is to return the diagonal entries (with multiplicity), bypassing Newton; this shortcut was disabled in our benchmark runner. For singular inputs, the solver estimates the right nullity $\nu_0:=\dim\Ker(A)=n-\rank(A)$ (equivalently via $\rank(\rho(A))=4\rank(A)$), pre-fills $\nu_0$ copies of
$\lambda=0$ together with a quaternion basis of $\Ker(A)$ (validated by the same certificate checks), and runs Newton only for the remaining $K-\nu_0$ eigenpairs.

Nonconvergence, convergence to an already detected eigenvalue, or insufficient residual reduction triggers a restart (a fresh trial). Trial/iteration statistics (iterations, trials, restarts, timing) are recorded; full tables and detailed logs are provided in Supplement~A.

\subsection{Reported measures}
\label{app:reported_metrics}
For each test matrix $A\in\HH^{n\times n}$ we request $K$ distinct left eigenpairs and obtain, after acceptance and de-duplication,
$K_{\rm found}$ accepted pairs $\{(\lambda_k,v_k)\}$ with $\|v_k\|_2=1$.
Accuracy is certified by the eigenpair residual
$\mathrm{res}_k:=\|Av_k-\lambda_k v_k\|_2$ and summarized by
$\mathrm{maxRes}(A):=\max_k \mathrm{res}_k$ and, over an ensemble,
$\mathrm{med}(\mathrm{maxRes}):=\mathrm{median}_A(\mathrm{maxRes}(A))$.
The acceptance rate is $\mathrm{acc}(A)(\%):=100\,K_{\rm found}/K$, and runtime is summarized by the mean time per accepted
eigenpair $\overline{t}_\lambda$ (ms). Additional diagnostic counters (trials, restarts, iterations, and per-trial statistics)
are reported in Supplement~A.

\subsection{Benchmark setup and success criteria}
\label{app:benchmark_setup}
Because the \emph{left} spectrum is not invariant under quaternionic similarity, we do not benchmark by prescribing spectra and applying similarity transforms. Instead, we test directly constructed matrix families that scale with $n$ and reflect standard structures in numerical linear algebra.

For each matrix size from $2\times2$ up to $64\times64$ and for each family, we generate a size-dependent number $n_{\rm Mat}$ of independent samples: we use very large ensembles for the smaller sizes and gradually reduce the sample count as $n$ increases, down to $n_{\rm Mat}=20$ at the largest size. For each matrix, we request $K=n$ distinct left eigenpairs. A run is declared \emph{successful} only if $K_{\rm found}=n$ after acceptance and de-duplication (equivalently, $\mathrm{acc}(A)(\%)=100$); otherwise the outcome is treated as a failure within the prescribed trial budget and is recorded for reproducibility in Supplement~A.

We consider four representative families: random upper-triangular matrices with random diagonal (a controlled sanity check, since diagonal entries must be left eigenvalues); dense i.i.d.\ Gaussian matrices; Hermitian matrices
$A=\tfrac12(R+R^\ast)$; and sparse matrices with fixed density $p=0.1$ obtained by masking i.i.d.\ Gaussian entries.
Supplement~A contains the complete per-size tables, detailed diagnostic statistics, benchmark figures, and the full experiment protocol.

Finally, spectrum-deficient cases (fewer than $n$ distinct isolated left eigenvalues) and continuous components (e.g., spherical families) are not targeted in this benchmark suite. They are handled by dedicated diagnostic modes and are reported separately in Supplement~B (Zoo) and, where relevant, in Supplement~A.

\subsection{Summary of benchmark observations}
\label{app:bench_takeaways}
The benchmark suite tests the solver as an end-to-end computational pipeline: multi-start Newton trials, acceptance by residual-based certificates, and de-duplication to obtain $K=n$ distinct eigenpairs per matrix. Across all tested sizes up to $n=64$ and across all four matrix families, this pipeline achieved full recovery in the benchmark sense on essentially all sampled instances, i.e., $\mathrm{acc}(A)(\%)=100$ (equivalently, $K_{\rm found}=n$). Supplement~A reports the complete per-size tables together with trial/iteration/restart diagnostics.

The upper-triangular family provides a controlled sanity check: the left eigenvalues are the diagonal entries, and the computed pairs attain residuals close to machine precision. For dense unstructured and Hermitian ensembles, the solver likewise returns eigenpairs with small certified residuals; the dominant cost is the sequence of embedded linear solves in the Newton corrections, together with occasional restarts when a trial converges to an already detected eigenvalue or fails to reduce the residual sufficiently.

Sparse inputs are the most variable: at smaller sizes, they may exhibit larger nullspaces and, more broadly, they require careful restart budgeting to maintain a high success rate. Nevertheless, the same acceptance-and-certificate mechanism remains effective, and Supplement~A documents both typical behavior and rare hard instances (including reproducibility captures).

Two caveats help interpret these results. First, the benchmark mode targets the generic situation of $n$ isolated left eigenvalues; it does not aim to systematically trigger non-isolated components (e.g., spherical families) or
left-spectrum-deficient behavior. When such phenomena occur, the implementation provides diagnostic modes, documented and illustrated in Supplement~B (Zoo). Second, if an application demands tighter accuracy, the refinement strategies of
Section~\ref{sec:refinement} typically improve residual-based certificates to (near) machine precision.

\section*{Supplementary material}
\label{sec:supplementary_material}
Supplementary material accompanies this article and is available online at the publisher's platform.
It is split into two parts.

\emph{Supplement~A (Implementation, experiments, and performance).}
This part documents the experimental protocol (including the precise definitions of all reported metrics), summarizes the tested matrix families together with the observed benchmark behavior, and provides the complete
per-size numerical tables and all benchmark figures.
It also reports implementation-facing diagnostics (trials, restarts, and iteration counts), describes a reproducibility-oriented problem-capture policy for rare hard instances, and includes a short refinement note and code-availability information.

\emph{Supplement~B (Zoo of interesting matrices).}
This part collects additional worked examples that arose during development and testing, including matrices with more than $n$ isolated left eigenvalues, left-spectrum deficient constructions (fewer than $n$ isolated eigenvalues), and spherical or mixed configurations, together with numerical certificates and reproduction-ready data.

\section*{Acknowledgements}
This work was co-funded by the European Union under the project RO\-BO\-PROX (reg. no. CZ.02.01.01/00/22\_008/0004590).

\section*{Declaration of competing interest}
None declared.

\section*{Declaration of generative AI and AI-assisted technologies in the manuscript preparation process}
During the preparation of this work, the author used ChatGPT \mbox{(OpenAI)}, a generative AI tool, in order to
(i) support the drafting and refinement of the related-work overview by suggesting search terms and
candidate thematic connections to relevant literature,
(ii) assist with MATLAB coding and debugging, including refactoring research code into a stand-alone
public reference implementation, and
(iii) improve the language, clarity, and readability of the manuscript.
After using this tool, the author reviewed, verified, and edited all content as needed, validated the software by testing, and takes full responsibility for the content of the published article.



\begin{thebibliography}{00}
\bibitem{ColomboSabadiniStruppa2011}
F.~Colombo, I.~Sabadini, D.~C.~Struppa,
\emph{Noncommutative Functional Calculus: Theory and Applications of Slice Hyperholomorphic Functions},
Progress in Mathematics, Vol.~289,
Birkh\"auser/Springer, Basel, 2011.
\DOI{10.1007/978-3-0348-0110-2}

\bibitem{DeLeoScolariciSolombrino2002}
S.~De~Leo, G.~Scolarici, and L.~Solombrino.
\newblock Quaternionic eigenvalue problem.
\newblock \emph{J.\ Math.\ Phys.}, 43(11):5815--5829, 2002.
\newblock \DOI{10.1063/1.1511789}.

\bibitem{DennisSchnabel1996}
J.~E.~Dennis and R.~B.~Schnabel.
\newblock \emph{Numerical Methods for Unconstrained Optimization and Nonlinear Equations}.
\newblock SIAM, Philadelphia, 1996 (reprint of 1983 ed.).

\bibitem{Deuflhard2011}
P.~Deuflhard.
\newblock \emph{Newton Methods for Nonlinear Problems: Affine Invariance and Adaptive Algorithms}.
\newblock Springer, Berlin, 2011.

\bibitem{FarenickPidkowich2003}
D.R.~Farenick, B.A.F.~Pidkowich,
\newblock The spectral theorem in quaternions,
\newblock \emph{Linear Algebra Appl.} 371 (2003) 75--102.
\newblock \DOI{10.1016/S0024-3795(03)00420-8}

\bibitem{GolubVanLoan2013}
G.~H. Golub and C.~F. Van Loan,
\emph{Matrix Computations}, 4th ed.,
Johns Hopkins University Press, Baltimore, 2013.

\bibitem{HuangSo2001}
L.~Huang and W.~So.
\newblock On left eigenvalues of a quaternionic matrix.
\newblock \emph{Linear Algebra Appl.}, 323(1--3):105--116, 2001.
\newblock \DOI{10.1016/S0024-3795(00)00246-9}.

\bibitem{LiuKou2019}
W.~Liu and K.~I.~Kou.
\newblock Quaternionic left eigenvalue problem: a matrix representation.
\newblock \emph{arXiv preprint} arXiv:1903.08897, 2019.
\newblock \DOI{10.48550/arXiv.1903.08897}.

\bibitem{MaciasVirgosPereiraSaez2009}
E.~Mac{\'i}as-Virg{\'o}s and M.~J.~Pereira-S{\'a}ez.
\newblock Left eigenvalues of $2\times 2$ symplectic matrices.
\newblock \emph{Electron.\ J.\ Linear Algebra}, 18:274--280, 2009.

\bibitem{MaciasVirgosPereiraSaez2014}
E.~Mac{\'i}as-Virg{\'o}s and M.~J.~Pereira-S{\'a}ez.
\newblock A topological approach to left eigenvalues of quaternionic matrices.
\newblock \emph{Linear Multilinear Algebra}, 62(2):139--158, 2014.
\newblock \DOI{10.1080/03081087.2012.753599}.

\bibitem{MaciasVirgosPereiraSaez2014b}
E.~Mac{\'i}as-Virg{\'o}s and M.~J.~Pereira-S{\'a}ez.
\newblock Cayley--Hamilton theorem for left eigenvalues of $3\times 3$ quaternionic matrices.
\newblock \emph{Special Matrices}, 2(1):11--18, 2014.
\newblock \DOI{10.2478/SPMA-2014-0002}.

\bibitem{NelderMead1965}
J.~A. Nelder and R.~Mead,
A simplex method for function minimization,
\emph{The Computer Journal} 7 (1965) 308--313.

\bibitem{OrtegaRheinboldt2000}
J.~M.~Ortega and W.~C.~Rheinboldt.
\newblock \emph{Iterative Solution of Nonlinear Equations in Several Variables}.
\newblock SIAM Classics in Applied Mathematics, SIAM, Philadelphia, 2000 (reprint).

\bibitem{PanNg2024}
J.~Pan and M.~K.~Ng,
\newblock Block-Diagonalization of Quaternion Circulant Matrices with Applications,
\newblock \emph{SIAM J. Matrix Anal. Appl.} 45(3):1429--1454, 2024.
\newblock \DOI{10.1137/23M1552115}.

\bibitem{Rodman2014}
L.~Rodman.
\newblock \emph{Topics in Quaternion Linear Algebra}.
\newblock Princeton Series in Applied Mathematics.
\newblock Princeton University Press, Princeton, NJ, 2014.
\newblock \DOI{10.1515/9781400852741}.

\bibitem{Ward1997}
J.~P.~Ward,
\emph{Quaternions and Cayley Numbers: Algebra and Applications},
Mathematics and Its Applications, Vol.~403,
Kluwer Academic Publishers, Dordrecht, 1997.

\bibitem{Wood1985}
R.~M.~W.~Wood.
\newblock Quaternionic eigenvalues.
\newblock \emph{Bull.\ London Math.\ Soc.}, 17(2):137--138, 1985.
\newblock \DOI{10.1112/blms/17.2.137}.

\bibitem{Zhang1997}
F.~Zhang.
\newblock Quaternions and matrices of quaternions.
\newblock \emph{Linear Algebra Appl.}, 251:21--57, 1997.
\newblock \DOI{10.1016/S0024-3795(95)00543-9}.

\bibitem{Zhang2007}
F.~Zhang.
\newblock Gerschgorin type theorems for quaternionic matrices.
\newblock \emph{Linear Algebra Appl.}, 424(2):1399--1413, 2007.
\newblock \DOI{10.1016/j.laa.2007.01.006}.

\end{thebibliography}
\end{document}